\documentclass{amsart}

\usepackage{fullpage}

\usepackage{amsmath}
\usepackage{amsfonts}
\usepackage{amssymb}
\usepackage{graphicx}
\usepackage{mathrsfs}
\usepackage{amsthm}
\usepackage{enumerate}
\usepackage{leftidx}
\usepackage{hyperref}
\usepackage{extarrows}
\usepackage[all]{xy}
\usepackage{stmaryrd}
\usepackage{wasysym}
\usepackage[OT2,T1]{fontenc}
\usepackage{epstopdf}
\usepackage {hyperref}
\numberwithin{equation}{subsubsection}
\newtheorem{thm}[subsubsection]{Theorem}
\newtheorem*{thm*}{Theorem}
\newtheorem*{thmA}{Theorem A}
\newtheorem*{thmB}{Theorem B}

\newtheorem{cor}[subsubsection]{Corollary}
\newtheorem{lem}[subsubsection]{Lemma}
\newtheorem{prop}[subsubsection]{Proposition}
\theoremstyle{definition}
\newtheorem{defn}[subsubsection]{Definition}
\theoremstyle{remark}
\newtheorem{rem}[subsubsection]{Remark}

\DeclareMathOperator{\anti}{anti}

\DeclareMathOperator{\ad}{ad}

\DeclareMathOperator{\Aut}{Aut}

\DeclareMathOperator{\cris}{cris}

\DeclareMathOperator{\diag}{diag}

\DeclareMathOperator{\End}{End}

\DeclareMathOperator{\Hom}{Hom}

\DeclareMathOperator{\Fil}{Fil}

\DeclareMathOperator{\im}{im}
\DeclareMathOperator{\Lie}{Lie}

\DeclareMathOperator{\Gr}{Gr}
\DeclareMathOperator{\OGr}{OGr}

\DeclareMathOperator{\triv}{triv}

\DeclareMathOperator{\Sh}{Sh}
\DeclareMathOperator{\id}{id}

\DeclareMathOperator{\spec}{Spec}

\DeclareMathOperator{\GL}{GL}
\DeclareMathOperator{\Gsp}{GSp}
\DeclareMathOperator{\GSp}{GSp}

\DeclareMathOperator{\SO}{SO}

\DeclareMathOperator{\Gspin}{GSpin}
\DeclareMathOperator{\GSpin}{GSpin}

\DeclareMathOperator{\RZ}{RZ}
\DeclareMathOperator{\BT}{BT}
\DeclareMathOperator{\Spf}{Spf}
\DeclareMathOperator{\VL}{VL}
\DeclareMathOperator{\ord}{ord}
\DeclareMathOperator{\inv}{inv}
\DeclareMathOperator{\nilp}{Nilp}
\DeclareMathOperator{\anilp}{ANilp^{fsm}_W}

\newcommand{\lprod}[1]{\langle#1\rangle}
\newcommand{\adele}{\mathbb A}
\newcommand{\aff}[1]{\mathbb A^{#1}}

\newcommand{\set}[1]{\left\{#1\right\}}

\newcommand{\To}{\longrightarrow}
\newcommand{\isom}{\overset{\sim}{\To}}

\newcommand{\NN}{\mathbb{N}}

\newcommand{\ZZ}{\mathbb{Z}}
\newcommand{\QQ}{\mathbb{Q}}

\newcommand{\FF}{\mathbb{F}}

\newcommand{\GG}{\mathbb{G}}
\newcommand{\oo}{\mathcal{O}}

\newcommand{\Shh}{\mathscr{S}}

\newcommand{\LL}{\mathcal L}

\newcommand{\ignore}[1]{}
\DeclareSymbolFont{cyrletters}{OT2}{wncyr}{m}{n}
\DeclareMathSymbol{\Sha}{\mathalpha}{cyrletters}{"58}

\usepackage{enumitem}

\setlist[enumerate]{leftmargin=*}
\setlist[itemize]{leftmargin=*}

\title{Arithmetic intersection on GSpin Rapoport--Zink spaces}

\author[Chao Li]{Chao Li}\email{chaoli@math.columbia.edu} 
\address{Department of Mathematics, Columbia University, 2990 Broadway,
  New York, NY 10027}

\author[Yihang Zhu]{Yihang Zhu}\email{yihang@math.harvard.edu}
\address{Department of Mathematics, Harvard University, 1 Oxford Street, Cambridge, MA 02138}


\date{\today}

\begin{document}

\maketitle

\begin{abstract}
  We prove an explicit formula for the arithmetic intersection number of diagonal cycles on GSpin Rapoport--Zink spaces in the minuscule case. This is a local problem arising from the arithmetic Gan--Gross--Prasad conjecture for orthogonal Shimura varieties. Our formula can be viewed as an orthogonal counterpart of the arithmetic-geometric side of the arithmetic fundamental lemma proved by Rapoport--Terstiege--Zhang in the minuscule case.
\end{abstract}

\setcounter{tocdepth}{1}
\tableofcontents{}

\section{Introduction}

\subsection{Motivation}

The \emph{arithmetic Gan--Gross--Prasad conjectures} (arithmetic GGP) generalize the celebrated Gross--Zagier formula to higher dimensional Shimura varieties (\cite[\S 27]{Gan2012}, \cite[\S 3.2]{Zhang2012}). It is a conjectural identity relating the heights of certain algebraic cycles on Shimura varieties to the central derivative of certain Rankin--Selberg $L$-functions. Let us briefly recall the rough statement of the conjecture. The diagonal embeddings of unitary groups $$H=\mathrm{U} (1,n-1)\hookrightarrow G=\mathrm{U}(1, n-1)\times \mathrm{U} (1, n)$$ or of orthogonal groups $$H=\SO(2, n-1)\hookrightarrow G=\SO(2,n-1)\times \SO(2,n),$$ induces an embedding of Shimura varieties $\Sh_H\hookrightarrow \Sh_G$. We denote its image by $\Delta$ and call it the \emph{diagonal cycle} or the \emph{GGP cycle} on $\Sh_G$. Let $\pi$ be a tempered cuspidal automorphic representation on $G$ appearing in the middle cohomology of $\Sh_G$. Let  $\Delta_\pi$ be the (cohomological trivialization) of the $\pi$-component of $\Delta$. The arithmetic GGP conjecture asserts that the (conditional) Beilinson--Bloch--Gillet--Soul\'e height of $\Delta_\pi$   should be given by the central derivative of a certain Rankin-Selberg $L$-function $L(s, \pi)$ up to simpler factors $$\langle \Delta_\pi,\Delta_\pi\rangle\sim L'(1/2, \pi).$$

The Gross--Zagier formula \cite{Gross1986} and the work of Gross, Kudla, Schoen (\cite{Gross1992}, \cite{Gross1995}) can be viewed as the special cases $n=1$ and $n=2$ in the orthogonal case correspondingly. The recent work of Yuan--Zhang--Zhang (\cite{Yuan2013}, \cite{Yuan})  has proved this conjecture for $n=1,2$ in the orthogonal case in vast generality.

In the unitary case, W. Zhang has proposed an approach for general $n$ using the relative trace formula of Jacquet--Rallis. The relevant arithmetic fundamental lemma relates an arithmetic intersection number of GGP cycles on unitary Rapoport--Zink spaces with a derivative of orbital integrals on general linear groups. The arithmetic fundamental lemma has been verified for $n=1,2$ (\cite{Zhang2012}) and for general $n$ in the minuscule case by Rapoport--Terstiege--Zhang \cite{RTZ}.

In the orthogonal case, very little is known currently beyond $n=1,2$ and no relative trace formula approach has been proposed yet. However it is notable that R. Krishna \cite{Krishna2016} has recently established a relative trace formula for the case $\SO(2)\times \SO(3)$ and one can hope that his method will generalize to formulate a relative trace formula approach for general $\SO(n-1)\times \SO(n)$.

Our goal in this article is to establish an orthogonal counterpart of the arithmetic-geometric side of the arithmetic fundamental lemma in \cite{RTZ}, namely to formulate and compute the arithmetic intersection of GGP cycles on $\GSpin$ Rapoport--Zink spaces in the minuscule case.

\subsection{The main results} Let $p$ be an odd prime. Let $k=\overline{\mathbb{F}}_p$, $W=W(k)$, $K=W[1/p]$ and $\sigma\in\Aut(W)$ be the lift of the absolute $p$-Frobenius on $k$. Let $n\ge4$. Let $V^\flat$ be a self-dual quadratic space over $\ZZ_p$ of rank $n-1$ and let $V=V^{\flat} \oplus \mathbb{Z}_p x_n$ (orthogonal direct sum) be a self-dual quadratic space over $\ZZ_p$ of rank $n$, where $x_n$ has norm 1. Associated to the embedding of quadratic spaces $V^\flat\hookrightarrow V$ we have an embedding of algebraic groups $G^\flat=\GSpin(V^\flat)\hookrightarrow G=\GSpin(V)$ over $\mathbb{Z}_p$. After suitably choosing compatible local unramified Shimura--Hodge data $(G^\flat, b^\flat, \mu^\flat, C(V^\flat))\hookrightarrow (G, b, \mu, C(V))$, we obtain a closed immersion of the associated $\GSpin$ Rapoport--Zink spaces $$\delta: \RZ^\flat\hookrightarrow \RZ.$$  See \S \ref{sec:gspin-rapoport-zink} for precise definitions and see \S \ref{sec:relat-with-spec} for the moduli interpretation of $\delta$.   The space $\RZ$ is an example of Rapoport--Zink spaces of Hodge type, recently constructed by Kim \cite{Kim2013} and Howard--Pappas \cite{Howard2015}. It is a formal scheme over $\Spf W$, parameterizing deformations of a $p$-divisible group $\mathbb{X}_0/k$ with certain crystalline Tate tensors (coming from the defining tensors of $G$ inside some $\GL_N$). Roughly speaking, if $X^\flat$ is the $p$-divisible group underlying a point $x\in \RZ^\flat$, then the $p$-divisible group underlying $\delta(x) \in \RZ$ is given by $X= X^{\flat} \oplus X^{\flat}$.

\begin{rem}
 The datum $(G,b,\mu, C(V))$ is chosen such that the space $\RZ$ provides a $p$-adic uniformization of $(\widehat {\mathscr{S}_W})_{/\mathscr{S}_{\mathrm{ss}}}$, the formal completion of $\mathscr{S}_W$ along $\mathscr{S}_{\mathrm{ss}}$, where  $\mathscr{S}_W$ is the base change to $W$ of Kisin's integral model (\cite{kisin2010integral}) of a $\GSpin$ Shimura variety (which is of Hodge type) at a good prime $p$, and $\mathscr{S}_{\mathrm{ss}}$ is the supersingular locus (= the basic locus in this case) 
  of the special fiber of $\Shh_W$ (see \cite[7.2]{Howard2015}).
\end{rem}

The group $J_b(\mathbb{Q}_p)=\{g\in G(K): gb=b\sigma(g)\}$ is the $\mathbb{Q}_p$-points of an inner form of $G$ and acts on $\RZ$ via its action on the fixed $p$-divisible group $\mathbb{X}_0$. Let $g\in J_b(\mathbb{Q}_p)$. As explained in \S \ref{sec:intersection-problem}, the intersection of the GGP cycle $\Delta$ on $\RZ^\flat\times_{W}\RZ$ and its $g$-translate leads to study of the formal scheme
\begin{equation}
  \label{eq:intersectionformalscheme}
  \delta(\RZ^\flat)\cap \RZ^g,
\end{equation}
 where $\RZ^g$ denotes the $g$-fixed points of $\RZ$.

 We call $g\in J_b(\mathbb{Q}_p)$ \emph{regular semisimple} if $$L(g):=\mathbb{Z}_px_n+\mathbb{Z}_p gx_n+ \cdots+\mathbb{Z}_p g^{n-1}x_n$$ is a free $\mathbb{Z}_p$-module of rank $n$. Let $L(g)^{\vee}$ denote the dual lattice of $L(g)$. We further call $g$ \emph{minuscule} if $L(g) \subset L(g)^{\vee} $ (i.e. the quadratic form restricted to $L(g)$ is valued in $\ZZ_p$), and $L(g)^\vee/L(g)$ is a $\mathbb{F}_p$-vector space. See Definition \ref{def:relative regular} for equivalent definitions. When $g\in J_b(\mathbb{Q}_p)$ is regular semisimple and minuscule, we will show that the formal scheme (\ref{eq:intersectionformalscheme})  is in fact a 0-dimensional scheme of characteristic $p$. Our main theorem is an explicit formula for its arithmetic intersection number  (i.e., the total $W$-length of its local rings).

To state the formula, assume $g$ is regular semisimple and minuscule, and suppose $\RZ^g$ is nonempty. Then $g$ stabilizes both $L(g)^\vee$ and $L(g)$ and thus acts on the $\mathbb{F}_p$-vector space $L(g)^\vee/L(g)$. Let $P(T)\in \mathbb{F}_p[T]$ be the characteristic polynomial of $g$ acting on $L(g)^\vee/L(g)$. For any irreducible polynomial $R(T)\in \mathbb{F}_p[T]$, we denote its multiplicity in $P(T)$ by $m(R(T))$. Moreover, for any polynomial $R(T)$, we define its \emph{reciprocal} by $$R^*(T):=T^{\deg R(T)}\cdot R(1/T).$$ We say $R(T)$ is \emph{self-reciprocal} if $R(T)=R^*(T)$. Now we are ready to state our main theorem:

\begin{thmA}\label{main thm}
  Let $g\in J_b(\mathbb{Q}_p)$ be regular semisimple and minuscule. Assume $\RZ^g$ is non-empty. Then
  \begin{enumerate}
  \item  (Corollary \ref{concentration}) $\delta(\RZ^\flat)\cap \RZ^g$ is a scheme of characteristic $p$.
  \item (Theorem \ref{thm:pointscounting}) $\delta(\RZ^\flat)\cap \RZ^g$ is non-empty if and only if $P(T)$ has a unique self-reciprocal monic irreducible factor $Q(T)|P(T)$ such that $m(Q(T))$ is odd. In this case, $p^\mathbb{Z}\backslash(\delta(\RZ^\flat)\cap\RZ^g)(k)$ is finite and has cardinality $$\deg Q(T)\cdot \prod_{R(T)}(1+m(R(T))),$$ where $R(T)$ runs over all non-self-reciprocal monic irreducible factors of $P(T)$. Here, the group $p^{\ZZ}$ acts on $\RZ$ via the central embedding $p^{\ZZ} \hookrightarrow J_b(\QQ_p)$, and the action stabilizes $\delta(\RZ^\flat)\cap\RZ^g$.
  \item (Corollary \ref{cor:mainformula}) Let $c = \frac{m(Q(T))+1}{2}$. Then $1\leq c \leq n/2$. Assume $p>c$. Then $\delta(\RZ^\flat)\cap \RZ^g$ is a disjoint union over its $k$-points of copies of $\spec k[X]/X^c$. In particular, the intersection multiplicity at each $k$-point of $\delta(\RZ^\flat)\cap \RZ^g$ is the same and equals $c$. 
  \end{enumerate}
\end{thmA}


Along the way we also prove a result that should be of independent interest. In \cite{Howard2015}, Howard--Pappas define closed formal subschemes $\RZ_{\Lambda}$ of $\RZ$ for each \textit{vertex lattice} $\Lambda$ (recalled in \S \ref{sec:gspin-rapoport-zink}). 
Howard--Pappas study the reduced subscheme $\RZ_{\Lambda}^{\mathrm{red}}$ detailedly and prove that they form a nice stratification of $\RZ^{\mathrm {red}}$. We prove:
\begin{thmB}[Theorem \ref{thm:reducedminusculecyle}]\label{minor thm}
	$\RZ _{\Lambda} = \RZ_{\Lambda}^{\mathrm{red}}$ for each vertex lattice $\Lambda$. 
      \end{thmB}

\subsection{Novelty of the method}
The results Theorem A and Theorem B are parallel to the results in \cite{RTZ} for unitary Rapoport--Zink spaces. The main new difficulty in the GSpin case is due to the fact that, unlike the unitary case, the GSpin Rapoport--Zink spaces are \textit{not} of PEL type. They are only of Hodge type, and as for now they lack full moduli interpretations that are easy to work with directly (see Remark \ref{rem:moduliinter}).

In \cite{RTZ}, the most difficult parts are the reducedness of minuscule special cycles \cite[Theorem 10.1]{RTZ} and the intersection length formula \cite[Theorem 9.5]{RTZ}. They are the analogues of Theorem B and Theorem A (3) respectively. In \cite{RTZ}, they are proved using Zink's theory of windows and displays of $p$-divisible groups and involve rather delicate linear algebra computation. In contrast, in our method we rarely directly work with $p$-divisible groups and we completely avoid computations with windows or displays. Instead we make use of what are essentially consequences of Kisin's construction of integral models of Hodge type Shimura varieties to \emph{abstractly} reduce the problem to algebraic geometry over $k$. More specifically, we reduce the intersection length computation to the study of a certain scheme of the form $S_\Lambda^{\bar g}$ (Proposition \ref{prop:nonreduced}), where $S_\Lambda$ is a smooth projective $k$-variety closely related to orthogonal Grassmannians, and $\bar g$ is a certain finite order automorphism of $S$.  Thus our method overcomes the difficulty of non-PEL type and also makes the actual computation much more elementary. 

It is worth mentioning that our method also applies to the unitary case considered in \cite{RTZ}. Even in this PEL type case, our method gives a new and arguably simpler proof of the arithmetic fundamental lemma in the minuscule case. This will be pursued in a forthcoming work.

It is also worth mentioning that the very recent work of Bueltel--Pappas \cite{Bueltel2017} gives a new moduli interpretation for  Rapoport--Zink spaces of Hodge type when restricted to $p$-nilpotent noetherian algebras. Their moduli description is purely group-theoretic (in terms of $(G,\mu)$-displays) and does not involve $p$-divisible groups. Although we do not use $(G,\mu)$-displays in this article, it would be interesting to see if it is possible to extend the results of this article using their group-theoretic description (e.g., to non-minuscule cases).

\subsection{Strategy of the proofs}

Our key observation is that in order to prove these theorems, we only need to understand $\mathcal{O}$-points of $\RZ$ for very special choices of $W$-algebras $\mathcal{O}$.

To prove Theorem B, it turns out that we only need to understand $\RZ(W/p^2)$ and $\RZ(k[\epsilon]/\epsilon^2)$. Note that the $W$-algebras $ W/p^2$ and $ k[\epsilon]/\epsilon^2$, when viewed as thickenings of $\spec k$ (under reduction modulo $p$ or $\epsilon$ respectively), are objects of the crystalline site of $\spec k$. For such an object $\oo$, we prove in Theorem \ref{thm: MP} an explicit description of $\RZ(\oo)$ and more generally an explicit description of $\mathcal Z(\oo) $, for any \textit{special cycle} $\mathcal Z$ in $\RZ$. Theorem \ref{thm: MP} is the main tool to prove Theorem B, and is also the only place we use $p$-divisible groups. This result is a Rapoport--Zink space analogue of a result of Madapusi Pera \cite[Proposition 5.16]{MPspin} for GSpin Shimura varieties. Its proof also relies on loc. cit. and is ultimately  a consequence of Kisin's construction of the integral canonical models of Hodge type Shimura varieties \cite{kisin2010integral}. 

To prove the intersection length formula Theorem A (3), let $\Lambda$ be the vertex lattice $L(g)^{\vee}$. Theorem B allows us to reduce Theorem A (3) to the problem of studying the fixed-point subscheme of  the smooth $k$-variety $S_\Lambda\cong p^\mathbb{Z}\backslash\RZ_{\Lambda}^{\mathrm{red}}$, under the induced action $\bar g\in \SO(\Lambda/\Lambda^\vee)$ of $g$. Since the fixed point of a smooth $k$-variety under a group of order coprime to $p$ is still smooth (\cite[1.3]{Iversen1972}), this point of view immediately explains that when $\bar g$ is semisimple (in which case $m(Q(T))=1$), the intersection multiplicity must be 1. More generally, we utilize Howard--Pappas's description of $S_\Lambda$ in \cite{Howard2015} and reduce the intersection length computation to elementary algebraic geometry of orthogonal Grassmannians over $k$ (Proposition \ref{structure of R} and Theorem \ref{main thm for multiplicity}).

The remaining parts of Theorem A are relatively easier. From Theorem B it is not difficult to deduce Theorem A (1). The set of $k$-points of $\RZ $ is well understood group theoretically in terms of the affine Deligne--Lusztig set. The point counting formula Theorem A (2) essentially only relies on this description, and we follow the strategy in \cite{RTZ} to give a short streamlined proof (Proposition \ref{prop:fixed-points}).



\subsection{Organization of the paper} In \S \ref{sec:gspin-rapoport-zink}, we review the structure of GSpin Rapoport--Zink spaces and special cycles. In \S \ref{sec:intersection-problem}, we formulate the arithmetic intersection problem of GGP cycles and prove the point-counting formula for the $k$-points of the intersection in the minuscule case (Theorem A (2)). In \S \ref{sec:reduc-minusc-spec}, we prove reducedness of minuscule special cycles (Theorem B). In \S \ref{sec:inters-length-form}, we deduce from Theorem B that the arithmetic intersection is concentrated in the special fiber (Theorem A (1)) and finally compute the intersection length when $p$ is sufficiently large (Theorem A (3)).

\subsection{Acknowledgments}  We are very grateful to B. Howard, M. Kisin, M. Rapoport and W. Zhang for helpful conversations or comments. Our debt to the two papers \cite{RTZ} and \cite{Howard2015} should be clear to the readers. 

\section{GSpin Rapoport--Zink spaces}\label{sec:gspin-rapoport-zink}

In this section we review the structure of GSpin Rapoport--Zink spaces due to Howard--Pappas \cite{Howard2015}. We refer to \cite{Howard2015} for the proofs of these facts.

\subsection{Quadratic spaces and GSpin groups}


Let $p$ be an odd prime. Let $(V,q)$ be a non-degenerate self-dual quadratic space over $\mathbb{Z}_p$ of rank $n\ge 3$. By definition the \emph{Clifford algebra} $C(V)$ is the quotient of the tensor algebra $V^{\otimes}$ by the two sided ideal generated by elements of the form $v \otimes v-q(v)$. It is free of rank $2^n$ over $\mathbb{Z}_p$. The linear map $v\mapsto -v$ preserves the quadratic form $q$ on $V$ and induces an involution on $C(V)$. This involution decomposes $C(V)=C^+(V) \oplus C^-(V)$ into even and odd parts. The image of the injection $V\hookrightarrow C^-(V)$ generates $C(V)$ as a $\ZZ_p$-algebra.

We also have a \emph{canonical involution} $*: C(V)\rightarrow C(V)$, which a $\mathbb{Z}_p$-linear endomorphism characterized by $(v_1v_2\cdots v_k)^*=v_k\cdots v_2v_1$ for $v_i\in V$. The \emph{spinor similitude group} $G=\GSpin(V)$ is the reductive group over $\mathbb{Z}_p$ such that for a $\mathbb{Z}_p$-algebra $R$, $$G(R)=\{g\in C^+(V)^\times: g V_R g^{-1}=V_R, \quad g^*g\in R^\times\}.$$ The character $\eta_G: G\rightarrow \mathbb{G}_m$ given by $g\mapsto g^*g$ is the called \emph{spinor similitude}.

The conjugation action  $g.v=gvg^{-1}$ of $G$ on $C(V)$  stabilizes $V$ and preserves the quadratic form $q$. This action thus defines a homomorphism $$G\rightarrow \SO(V).$$ The kernel of the above morphism is the central $\mathbb{G}_m$ inside $G$ given by the natural inclusion $R^{\times} \subset G(R)$ for any $\ZZ_p$-algebra $R$. The restriction of $\eta_G$ on the central $\mathbb{G}_m$ is given by $g\mapsto g^2$.
Note that the central $\GG_m$ in $G$ is equal to the identity component of the center of $G$, and it is equal to the center of $G$ precisely when $n$ is odd.

\subsection{Basic elements in GSpin groups}
\label{sec:basic-elements-gspin}
Let $k=\overline{\mathbb{F}}_p$, $W=W(k)$ and $K=W[1/p]$. Let $\sigma\in\Aut(W)$ be the lift of the absolute $p$-Frobenius on $k$.  Let $D=\Hom_{\mathbb{Z}_p}(C(V), \mathbb{Z}_p)$ be the contragredient $G$-representation of $C(V)$.

Any $b\in G(K)$ determines two isocrystals $$(V_K, \Phi=b\circ \sigma),\quad (D_K, F=b\circ \sigma).$$ Denote by $\mathbb T$ the pro-torus over $\QQ_p$ of character group $\QQ$. Recall that $b\in G(K)$ is \emph{basic} if its slope morphism $\nu_b: \mathbb{T}_K\rightarrow G_K$ factors through (the identity component) of $Z(G_K)$, i.e., factors through the central $\mathbb{G}_m$. By \cite[4.2.4]{Howard2015}, $b$ is basic if and only if $(V_K,\Phi)$ is isoclinic of slope 0, if and only if $(D_K, F)$ is isoclinic of slope $-\nu_b\in\Hom(\mathbb{T}_K,\mathbb{G}_m)\cong \mathbb{Q}.$ The map $b\mapsto \nu_b$ gives a bijection between the set of basic $\sigma$-conjugacy classes and the set $\frac{1}{2}\mathbb{Z}$. 
Moreover, the $\mathbb{Q}_p$-quadratic space $$V_K^\Phi=\{x\in V_K: \Phi x=x\}$$ has the same dimension and determinant as $V_{\mathbb{Q}_p}$, and has Hasse invariant $(-1)^{2\nu_b}$ (\cite[4.2.5]{Howard2015})).

\subsection{Local unramified Shimura--Hodge data}
\label{explicit choices}
 Since $V$ is self-dual, we know that $V_{\mathbb{Q}_p}$ has Hasse invariant $+1$. 
 In particular $V$ contains at least one hyperbolic plane 
 and  we can pick a $\mathbb{Z}_p$-basis $x_1,\ldots,x_n$ of $V$ such that the Gram matrix of the quadratic form $q$ is of the form $$
 \begin{pmatrix}
   0 & 1 & & &\\ 1& 0 & & & \\ & & * & & & \\ & & & * & & \\ & & & &\ddots & \\ & &  & & & *
 \end{pmatrix}
 $$  We will fix $x_1,\ldots, x_n$ once and for all. Define a cocharacter $$\mu: \mathbb{G}_m\rightarrow G, \quad t\mapsto t^{-1}x_1x_2+x_2x_1.$$  Pick an explicit element $b=x_3(p^{-1}x_1+x_2)\in G(\mathbb{Q}_p)$, then one can show that $b$ is basic with $\nu_b=\frac{1}{2}$. Thus $V_K^\Phi$ has the opposite Hasse invariant $-1$ (cf. \S \ref{sec:basic-elements-gspin}).

 Fix any $\delta\in C(V)^\times$ such that $\delta^*=-\delta$. Then $\psi_\delta(c_1,c_2)=\mathrm{Trd}(c_1\delta c_2^*)$ defines a non-degenerate symplectic form on $C(V)$, where $\mathrm{Trd}: C(V)\rightarrow \mathbb{Z}_p$ is the reduced trace. We have a closed immersion into the symplectic similitude group $$G\hookrightarrow \GSp(C(V), \psi_\delta).$$ By \cite[4.2.6]{Howard2015}, the tuple $(G, b, \mu, C(V))$ defines a \emph{local unramified Shimura--Hodge datum} (in the sense of \cite[2.2.4]{Howard2015}). In fact, for the fixed $G$ and $\mu$,  the $\sigma$-conjugacy class of $b$ is the unique basic $\sigma$-conjugacy class for which  $(G, b, \mu)$ is a local unramified Shimura--Hodge datum (cf. \cite[4.2.7]{Howard2015}).

 \begin{rem}
   The tuple $(G, b, \mu, C(V))$ is chosen in such a way that the associated Rapoport--Zink space (see below) provides a $p$-adic uniformization for the supersingular locus of a related $\Gspin$ Shimura variety. For more details on the relation with Shimura varieties see \cite[\S 7]{Howard2015}. 
 \end{rem}
 
\subsection{GSpin Rapoport--Zink spaces}\label{sec:gspin-rapoport-zink-spaces}
There is a unique (up to isomorphism) $p$-divisible group $\mathbb X_0/k$ such that its (contravariant) Dieudonn\'e module $\mathbb{D}(\mathbb{X}_0)$ is given by the $W$-lattice $D_W$ in the isocrystal $D_K$. The non-degenerate symplectic form $\psi_\delta$ induces a principal polarization $\lambda_0$ of $\mathbb{X}_0$.  Fix a collection of tensors $(s_\alpha)$ on $C(V)$ cutting out $G$ from $\GL(C(V))$ (including the symplectic form $\psi_\delta$). By \cite[4.2.7]{Howard2015}, we have a \emph{GSpin Rapoport--Zink space} $$\RZ:=\RZ(G, b,\mu, C(V), (s_\alpha)).$$ It is a formal scheme over $W$, together with a closed immersion into the symplectic Rapoport--Zink space $\RZ(\mathbb{X}_0,\lambda_0)$. Moreover, the formal scheme $\RZ$ itself depends only on the local unramified Shimura--Hodge datum $(G,b,\mu, C(V))$, and not on the choices of the tensors $(s_\alpha)$.

Denote by $(X, \rho ,\lambda)$ the universal triple over $\RZ(\mathbb X_0 , \lambda_0)$, where $X$ is the universal $p$-divisible group, $\rho$ is the universal quasi-isogeny, and $\lambda$ is the universal polarization. Consider the restriction of this triple to the closed formal subscheme $\RZ$ of $\RZ(\mathbb X_0 ,\lambda_0)$. We denote this last triple also by $(X,\rho ,\lambda)$ and call it the universal triple over $\RZ$.

\begin{rem}\label{rem:moduliinter}
Let $\nilp_W$ be the category of $W$-algebras in which $p$ is nilpotent. 
As a set-valued functor on the category $\nilp_W$, the symplectic Rapoport-Zink space $\RZ(\mathbb X_0  , \lambda_0)$ has an explicit moduli interpretation in terms of triples $(X,\rho,\lambda)$. In contrast, the subfunctor defined by $\RZ$ does not have an explicit description. In fact, in \cite{Howard2015} Howard--Pappas only give a moduli interpretation of $\RZ$ when it is viewed as a set-valued functor on a more restricted category $\anilp$. In this article we do not make use of this last moduli interpretation. All we will need is the \emph{global construction} of $\RZ$ as a formal subscheme of $\RZ(\mathbb X_0, \lambda_0)$ due to Howard--Pappas. 
\end{rem}


Over $\RZ$, the universal quasi-isogeny $\rho$ respects the polarizations $\lambda$ and $\lambda_0$ up to a scalar $c(\rho)\in \mathbb{Q}_p^\times$ , i.e., $\rho^\vee\circ \lambda\circ\rho=c^{-1}(\rho)\cdot \lambda_0$ (Zariski locally on $\RZ_k$). Let $\RZ^{(\ell)}\subseteq\RZ$ be the closed and open formal subscheme where $\ord_p(c(\rho))=\ell$. 
We have the decomposition into a disjoint union $$\RZ=\coprod_{\ell\in \mathbb{Z}}\RZ^{(\ell)}.$$ In fact each $\RZ^{(\ell)}$ is connected and they are mutually (non-canonically) isomorphic. cf. \cite[4.3.3, 4.3.4]{Howard2015}.

\subsection{The group $J_b$}

The algebraic group $J_b=\GSpin(V_K^\Phi)$ has $\mathbb{Q}_p$-points $$J_b(\mathbb{Q}_p)=\{g\in G(K): gb=b\sigma(g)\},$$ and $J_b(\QQ_p)$ acts on $\RZ$ via its action on $\mathbb X_0$ as quasi-endomorphisms. The action of $g \in J_b(\QQ_p)$ on $\RZ$ restricts to isomorphisms 
\begin{align}\label{g and l}
\RZ^{(\ell)} \isom \RZ^{(\ell + \ord_p (\eta_b (g)))}, ~ \ell \in \ZZ
\end{align}
 where $\eta_b : J_b(\QQ_p) \to \QQ_p^{\times}$ is the spinor similitude. In particular, $p^\mathbb{Z}\subseteq J_b(\mathbb{Q}_p)$ acts on $\RZ$ and since $\eta_b (p) = p^2$, we have an isomorphism $$p^\mathbb{Z}\backslash\RZ\cong \RZ^{(0)}\coprod\RZ^{(1)}.$$

\begin{rem}
	In this article we are interested in studying the fixed locus $\RZ^g$ of $\RZ$ under $g \in J_b (\QQ_p)$. By (\ref{g and l}) this is non-empty only when $\ord_p (\eta_b (g)) = 0 $. Since $p^{\ZZ}$ is central in $J_b(\QQ_p)$, one could also study $(p^{\ZZ}\backslash \RZ)^g$ for $g\in J_b(\QQ_p)$. However by (\ref{g and l}), we know that $(p^{\ZZ}\backslash \RZ)^g \neq \varnothing$ only if $\ord_p (\eta_b (g))$ is even, and in this case $$ (p^{\ZZ}\backslash \RZ)^g \cong p^{\ZZ} \backslash \RZ^{g_0},$$ where $g_0=p^{-\ord_p (\eta_b (g))/2} g$. Hence the study of $(p^{\ZZ}\backslash \RZ)^g$ for general $g$ reduces to the study of $\RZ^g$ for $g$ satisfying $\ord_p(\eta_b(g))=0$.
\end{rem}

\subsection{Special endomorphisms}\label{sec:spec-endom}

Using the injection $V\hookrightarrow C(V)^\mathrm{op}$, we can view $$V\subseteq \End_{\mathbb{Z}_p}(D)$$ as \emph{special endomorphisms} of $D$: the action of $v\in V$ on $D$ is explicitly given by $$(vd)(c)=d(vc),\quad d\in D, c\in C(V).$$ 
Base changing to $K$ gives $V_K\subseteq \End_K(D_K)$. Since the $F$-equivariant endomorphisms $\End_{K,F}(D_K)$ can be identified with the space of quasi-endomorphisms $\End^0(\mathbb X_0)$ of $\mathbb{X}_0$, we obtain an embedding of $\QQ_p$-vector spaces $$V_K^\Phi\hookrightarrow \End^0(\mathbb X_0).$$ Elements of $V_K^{\Phi}$ are thus viewed as quasi-endomorphisms of $\mathbb X_0$, and we call them \emph{special quasi-endomorphisms}.

\subsection{Vertex lattices}\label{sec:vertex-lattices}

\begin{defn}\label{def: vertex lattice}
A \emph{vertex lattice} is a $\mathbb{Z}_p$-lattice $\Lambda\subseteq V_K^\Phi$ such that $$p\Lambda\subseteq \Lambda^\vee\subseteq \Lambda.$$ We define $$\Omega_0=\Lambda/\Lambda^\vee.$$ Then the quadratic form $v\mapsto p\cdot q(v)$ makes $\Omega_0$ a non-degenerate quadratic space over $\mathbb{F}_p$. The \emph{type} of $\Lambda$ is defined to be $t_\Lambda:=\dim_{\mathbb{F}_p}\Omega_0.$
\end{defn}
By \cite[5.1.2]{Howard2015}, the type of a vertex lattice is always an \textit{even} integer such that $2\le t_\Lambda\le t_\mathrm{max}$ , where $$t_\mathrm{max}=
\begin{cases}
  n-2, & \text{if }n\text{ is even and }\det (V_{\mathbb{Q}_p})=(-1)^{n/2}\in \mathbb{Q}_p^\times/(\mathbb{Q}_p^\times)^2, \\
  n-1, & \text{if }n\text{ is odd}, \\
  n, & \text{if }n\text{ is even and }\det(V_{\mathbb{Q}_p})\ne(-1)^{n/2} \in \mathbb{Q}_p^\times/(\mathbb{Q}_p^\times)^2.
\end{cases}$$
It follows that the quadratic space $\Omega_0$ is always \emph{non-split}, because otherwise a Lagrangian subspace $\mathcal{L}\subseteq\Omega_0$ would provide a vertex lattice $\Lambda^\vee+\mathcal{L}\subseteq V_K^\Phi$ of type 0 (cf. \cite[5.3.1]{Howard2015}) 

\subsection{The variety $S_\Lambda$}
\label{sec:variety-s_lambda}

\begin{defn}
  Define $$\Omega=\Omega_0 \otimes_{\mathbb{F}_p} k\cong\Lambda_W/\Lambda_W^\vee.$$ Let $d=t_\Lambda/2$. Let $\OGr(\Omega)$ be the moduli space of Lagrangian subspaces $\mathcal{L}\subseteq \Omega$. We define $S_\Lambda\subseteq \OGr(\Omega)$ to be the reduced closed subscheme of $\OGr(\Omega)$ with $k$-points given as follows:
\begin{align*}
  S_\Lambda(k)&=\{\text{Lagrangian subspaces } \mathcal{L}\subseteq \Omega: \dim(\mathcal{L}+\Phi(\mathcal{L}))=d+1\} \\
  &\cong \{(\mathcal{L}_{d-1},\mathcal{L}_d): \mathcal{L}_d \subseteq \Omega \text{ Lagrangian}, \mathcal{L}_{d-1}\subseteq \mathcal{L}_d\cap\Phi\mathcal{L}_d,\dim \mathcal{L}_{d-1}=d-1\},
\end{align*}
where the last bijection is given by $\mathcal L \mapsto (\mathcal L \cap \Phi \mathcal L, \mathcal L)$.
\end{defn}

More precisely, for any $k$-algebra $R$, the $R$-points $S_{\Lambda} (R)$ is the set of pairs $(\LL_{d-1}, \LL_d)$ such that:
\begin{itemize}
\item $\LL_d$ is a totally isotropic $R$-submodule of $\Omega\otimes_k R$ that is an $R$-module local direct summand of $\Omega \otimes_k R$ and of local rank $d$,
\item $\LL_{d-1}$ is an $R$-module local direct summand of $\Omega \otimes_k R$ and of local rank $d-1$,
\item  $\LL_{d-1} \subset \LL_d \cap \Phi \LL_d$, where $\Phi$ acts on $\Omega\otimes_k R$ via the $p$-Frobenius on $R$. In particular, $\LL_{d-1}$ is totally isotropic, and is a local direct summand of $\LL_{d}$ and of $\Phi \LL_d$. (For the last statement see Remark \ref{rem: direct summand} below.)
\end{itemize}

By \cite[5.3.2]{Howard2015}, $S_\Lambda$ is a $k$-variety with two isomorphic connected components $S_\Lambda^{\pm}$, each being projective and smooth of dimension $t_\Lambda/2-1$. 
For more details, see \cite[\S 5.3]{Howard2015} and \cite[\S 3.2]{HPGU22}.

\begin{rem}\label{rem: direct summand}
	In the sequel we will frequently use the following simple fact without explicitly mentioning it. Let $R$ be a commutative ring and $M$ a free $R$-module of finite rank. Suppose $M_1, M_2$ are submodules of $M$ that are local direct summands. Suppose $M_1 \subset M_2$. Then $M_1$ is a local direct summand of $M_2$, and both $M_1$ and $M_2$ are locally free.   
\end{rem}

\subsection{Structure of the reduced scheme $\RZ^\mathrm{red}$}\label{sec:str}

\begin{defn}\label{def:RZLambda}
 For a vertex lattice $\Lambda$, we define $\RZ_\Lambda\subseteq \RZ$ to be locus where $\rho\circ\Lambda^\vee\circ\rho^{-1}\subseteq \End(X)$, i.e. the quasi-endomorphisms $\rho\circ v\circ\rho^{-1}$ lift to actual endomorphisms for any $v\in \Lambda^\vee$. In other words, if we define a locus $\RZ(\mathbb X_0,\lambda_0)_{\Lambda}$ using the same condition inside $\RZ(\mathbb X_0, \lambda_0)$ (a closed formal subscheme by \cite[Proposition 2.9]{RZ96}), then $\RZ_{\Lambda}$ is the intersection of $\RZ$ with $\RZ(\mathbb X_0,\lambda_0)_\Lambda$ inside $\RZ(\mathbb{X}_0,\lambda_0)$. In particular, $\RZ_{\Lambda} $ is a closed formal subscheme of $\RZ$. 
\end{defn}

  Consider the reduced subscheme $\RZ^{(\ell) , \mathrm{red}}$ of $\RZ^{(\ell)}$. 
By the result \cite[6.4.1]{Howard2015}, the irreducible components of $\RZ^{(\ell),\mathrm{red}}$ are precisely $\RZ_\Lambda^{(\ell),\mathrm{red}}$, where $\Lambda$ runs through the vertex lattices of the maximal type $t_\Lambda=t_\mathrm{max}$. Moreover, there is an isomorphism of $k$-schemes (\cite[6.3.1]{Howard2015})
\begin{equation}
  \label{eq:rzlambda}
  p^\mathbb{Z}\backslash\RZ_\Lambda^{\mathrm{red}}\isom S_\Lambda,  
\end{equation}
 which also induces an isomorphism between $\RZ_\Lambda^{(\ell),\mathrm{red}}$ and $S_\Lambda^\pm$, for each $\ell \in \ZZ$. In particular, $\RZ^\mathrm{red}$ is equidimensional of dimension $t_\mathrm{max}/2-1$.

\subsection{The Bruhat--Tits stratification}\label{sec:bruh-tits-strat} For any vertex lattices $\Lambda_1$ and $\Lambda_2$, the intersection $\RZ_{\Lambda_1}^\mathrm{red}\cap \RZ_{\Lambda_2}^\mathrm{red}$ is nonempty if and only if $\Lambda_1\cap \Lambda_2$ is also a vertex lattice, in which case it is equal to $\RZ_{\Lambda_1\cap \Lambda_2}^\mathrm{red}$ (\cite[6.2.4]{Howard2015}). In this way we obtain a \emph{Bruhat--Tits stratification} on $\RZ^\mathrm{red}$. Associated to a vertex lattice $\Lambda$, we define an open subscheme of $\RZ_\Lambda^\mathrm{red}$ given by $$\BT_\Lambda=\RZ^\mathrm{red}_\Lambda-\bigcup_{\Lambda'\subsetneq \Lambda}\RZ_{\Lambda'}^\mathrm{red}.$$  Then $$\RZ^\mathrm{red}=\coprod_{\Lambda}\BT_\Lambda$$ is a disjoint union of locally closed subschemes, indexed by all vertex lattices.

\subsection{Special lattices} One can further parametrize the $k$-points in each $\RZ_\Lambda$ using special lattices.

\begin{defn}
  We say a $W$-lattice $L\subseteq V_K$ is a \emph{special lattice} if $L$ is self-dual and $(L+\Phi(L))/L\cong W/pW$. 
\end{defn}

We have a bijection (\cite[6.2.2]{Howard2015}) \begin{equation}
  \label{eq:RZk}
p^\mathbb{Z} \backslash \RZ(k)\isom\{\text{special lattices } L\subseteq V_K\}.  
\end{equation}
 To construct this bijection, one uses the fact (\cite[3.2.3]{Howard2015}) that $p^\mathbb{Z} \backslash \RZ(k)$ can be identified with the affine Deligne--Lusztig set
 \begin{equation}
   \label{eq:afdl}
   X_{G,b,\mu^\sigma}(k)=\{g\in G(K): g^{-1}b \sigma(g)\in G(W)\mu^\sigma(p)G(W)\}/G(W).
 \end{equation}
The special lattice associated to $g\in G(K)$ is then given by $g\mu(p^{-1}).V_W\subseteq V_K$. Conversely, given a special lattice $L\subseteq V_K$, then there exists some $g\in G(K)$ such that $g\mu(p^{-1}). V_W=L$ and $g. V_W=\Phi(L)$. The point in $\RZ(k)$ then corresponds to the image of $g$ in $X_{G,b,\mu^\sigma}(k)$. The Dieudonn\'e module of the $p$-divisible group at this point is given by $M=g D_W\subseteq D_K$ and the image of Verschiebung is $(F^{-1} p) M=g\cdot p\mu(p^{-1}) D_W$.
\begin{rem}\label{moduli inter of L}
	Suppose $x_0\in \RZ(k)$ corresponds to the special lattice $L$ under (\ref{eq:RZk}). Let $M=\mathbb{D}(X_0) \subset D_K$ be the Dieudonn\'e module of the $p$-divisible group $X_0$ corresponding to $x_0$. 
        Then we have (cf. \cite[\S 6.2]{Howard2015}) $$L = \set{v\in V_K | v (F^{-1}p)M \subset (F^{-1}p) M },\quad
	\Phi L = \set{v\in V_K| v M \subset M}.$$ Here we view $V_K \subset \End_K(D_K)$ as in \S \ref{sec:spec-endom}.
\end{rem}

\subsection{Special lattices and vertex lattices}\label{sec:spec-latt-vert}

For any vertex lattice $\Lambda$, the bijection (\ref{eq:RZk}) induces a bijection 
\begin{align}\label{special lattices in RZLambda}
p^\mathbb{Z} \backslash \RZ_\Lambda(k)\isom\{\text{special lattices } L\subseteq V_K: \Lambda_W^\vee \subseteq L \subseteq \Lambda_W\} = \{\text{special lattices } L\subseteq V_K: \Lambda_W^\vee \subseteq L \}
\end{align} Sending a special lattice $L$ to $\mathcal{L}: =L/\Lambda_W^\vee\subseteq \Omega$ gives a bijection between the right hand side of (\ref{special lattices in RZLambda}) and $S_\Lambda(k)$, which is the effect of the isomorphism (\ref{eq:rzlambda}) on $k$-points.

\begin{defn}
For each special lattice $L\subseteq V_K$, there is a unique minimal vertex lattice $\Lambda(L)\subseteq V_K^\Phi$ such that $$\Lambda(L)_W^\vee\subseteq L\subseteq \Lambda(L)_W.$$ In fact, let $L^{(r)}=L+\Phi(L)+\cdots+\Phi^r(L)$. Then there exists a unique integer $1\le d\le t_\mathrm{max}/2$ such that $L^{(i)}\subsetneq L^{(i+1)}$ for $i<d$, and $L^{(d)}=L^{(d+1)}$. Then $L^{(i+1)}/L^{(i)}$ all have $W$-length 1 for $i<d$, and $$\Lambda(L):=(L^{(d)})^\Phi\subseteq V_K^\Phi$$ is a vertex lattice of type $2d$ and $\Lambda(L)^\vee=L^\Phi$.  
\end{defn}
 Notice that $\Lambda(L)_W$ is the smallest $\Phi$-invariant lattice containing $L$ and $\Lambda(L)^\vee_W$ is the largest $\Phi$-invariant lattice contained in $L$. It follows that the element of $\RZ(k)$ corresponding to a special lattice $L$ lies in $\RZ_\Lambda$ if and only if $\Lambda(L)\subseteq\Lambda$, and it lies in $\BT_\Lambda$ if and only if $\Lambda(L)=\Lambda$. Thus we have the bijection
\begin{equation}
  \label{eq:bt}
p^\mathbb{Z}\backslash\BT_\Lambda(k)\isom\{L \text{ special lattices}: \Lambda(L)=\Lambda\}.
\end{equation}

\subsection{Deligne--Lusztig varieties}\label{sec:deligne-luszt-vari}

For any vertex lattice $\Lambda$, by \cite[6.5.6]{Howard2015},  $p^\mathbb{Z}\backslash\BT_\Lambda$ is a smooth quasi-projective variety of dimension $t_\Lambda/2-1$, isomorphic to a disjoint union of two Deligne--Lusztig varieties $X_B(w^\pm)$ associated to two Coxeter elements $w^\pm$ in the Weyl group of $\SO(\Omega_0)$. Here $\Omega_0 := \Lambda/\Lambda^{\vee}$ is the quadratic space over $\FF_p$ defined in Definition \ref{def: vertex lattice}. In particular, the $k$-variety $p^\mathbb{Z}\backslash\BT_\Lambda$ only depends on the quadratic space $\Omega_0$.

Let us recall the definition of $X_B(w^\pm)$. Let $d=t_\Lambda/2$. Let $\lprod{\cdot , \cdot}$ be the bilinear pairing on $\Omega_0$. Since $\Omega_0$ is a non-degenerate non-split quadratic space over $\mathbb{F}_p$ (\S \ref{sec:vertex-lattices}), one can choose a basis  $e_1,\ldots, e_d, f_d,\ldots, f_1$ of $\Omega$ such that $\langle e_i,f_i\rangle=1$ and all other pairings between the basis vectors are 0, and $\Phi$ fixes $e_i,f_i$ for $i=1,\ldots,d-1$ and interchanges $e_d$ with $f_d$. This choice of basis gives a maximal $\Phi$-stable torus $T\subseteq \SO(\Omega)$ (diagonal under this basis), and a $\Phi$-stable Borel subgroup $B\supseteq T$ as the common stabilizer of the two complete isotropic flags $$\mathcal{F}^\pm: \langle e_1\rangle\subseteq \langle e_1,e_2\rangle\subseteq \cdots\subseteq\langle e_1, \ldots,e_{d-1}, e_d^\pm\rangle,$$ where $e_d^+:=e_d$ and $e_d^-:=f_d$. Let $s_i$ ($i=1,\ldots,d-2$) be the reflection $e_i\leftrightarrow e_{i+1}$, $f_i\leftrightarrow f_{i+1}$ and let $t^\pm$ be the reflection $e_{d-1}\leftrightarrow e_d^\pm$, $f_{d-1}\leftrightarrow e_d^\mp$. Then the Weyl group $W(T)=N(T)/T$ is generated by $s_1, \cdots, s_{d-2}, t^+, t^-$. We also know that $W(T)$ sits in a split exact sequence $$0\rightarrow (\mathbb{Z}/2 \mathbb{Z})^{d-1}\rightarrow W(T)\rightarrow S_d\rightarrow0.$$ Since $\Phi$ fixes $s_i$ and swaps $t^+$ and $t^-$, we know the $d-1$ elements $s_1,\ldots,s_{d-2}, t^+$ (resp. $s_1, \ldots, s_{d-2}, t^-$) form a set of representatives of $\Phi$-orbits of the simple reflections. Therefore $$w^\pm:= t^\mp s_{d-2}\cdots s_2s_1\in W(T)$$ are Coxeter elements of minimal length. The \emph{Deligne--Lusztig variety} associated to $B$ and the Coxeter element $w^\pm$ is defined to be $$X_B(w^\pm):=\{g\in \SO(\Omega)/B: \inv(g, \Phi(g))=w^\pm\},$$ where $\inv(g,h)\in  B\backslash \SO(\Omega)/B\cong W(T)$ is the relative position between the two Borels $gBg^{-1}$ and $hBh^{-1}$. The variety $X_B(w^\pm)$ has dimension $d-1$. Under the map $g\mapsto g\mathcal{F}^\pm$, the disjoint union $X_B(w^+)\coprod X_B(w^-)$ can be identified with the variety of complete isotropic flags $$\mathcal{F}: \mathcal{F}_1\subseteq \mathcal{F}_2\subseteq\cdots\subseteq \mathcal{F}_d$$ such that $\mathcal{F}_i=\mathcal{F}_{i-1}+\Phi(\mathcal{F}_{i-1})$ and $\dim_k( \mathcal{F}_d+\Phi(\mathcal{F}_d))=d+1$. The two components are interchanged by an orthogonal transformation of determinant $-1$. Notice that such $\mathcal{F}$ is determined by the isotropic line $\mathcal{F}_1$ by $$\mathcal{F}_i=\mathcal{F}_1+\Phi(F_1)\cdots+\Phi^{i-1}(\mathcal{F}_1),$$ and is also determined by the Lagrangian $\mathcal{F}_d$ by $$\mathcal{F}_i=\mathcal{F}_d\cap \Phi(\mathcal{F}_d)\cap\dots\cap\Phi^{d-i}(\mathcal{F}_d).$$

The bijection (\ref{eq:bt}) induces a bijection 
\begin{equation}
  \label{eq:btdlkpoints}
  p^\mathbb{Z}\backslash\BT_\Lambda(k)\isom X_B(w^+)(k)\coprod X_B(w^-)(k)
\end{equation}
by sending a special lattice $L$ with $\Lambda(L)=\Lambda$ to the flag determined by the Lagrangian $\mathcal{F}_d=L/\Lambda^\vee_W$. This bijection is the restriction of the isomorphism (\ref{eq:rzlambda}) on $k$-points and we obtain the desired isomorphism 
\begin{equation}
  \label{eq:btdl}
  p^\mathbb{Z}\backslash\BT_\Lambda\cong X_B(w^+)\coprod X_B(w^-).
\end{equation}

\subsection{Special cycles}


\begin{defn}\label{def:specialcycle}
  For an $m$-tuple $\mathbf{v}=(v_1,\ldots,v_m)$ of vectors in $V_K^\Phi$, define its \emph{fundamental matrix} $T(\mathbf{v})=(\langle v_i, v_j\rangle)_{i,j=1,\ldots, m}$. We define the \emph{special cycle} $\mathcal{Z}(\mathbf{v})\subseteq \RZ$ to be the locus where $\rho\circ v_i\circ \rho^{-1}\in\End(X)$, i.e., all the quasi-endomorphisms $\rho\circ v_i\circ \rho^{-1}$ lift to actual endomorphisms on $X$ ($i=1,\ldots,m$). Similar to Definition \ref{def:RZLambda}, $\mathcal Z(\mathbf v)$ is a closed formal subscheme of $\RZ$, which is the intersection $\RZ$ with the analogously defined cycle inside $\RZ(\mathbb{X}_0,\lambda_0)$. Since $\mathcal Z(\mathbf v)$ only depends on the $\ZZ_p$-submodule $ \mathrm{span}_{\ZZ_p} (\mathbf v)$ of $V_K^{\Phi}$, we also write $\mathcal Z( \mathrm{span}_{\ZZ_p} (\mathbf v) )$.  
\end{defn}
\begin{rem}\label{rem: v and L}
	Let $x_0 \in \RZ(k)$ correspond to $L $ under (\ref{eq:RZk}). Let $\mathbf v$ be an arbitrary $\ZZ_p$-submodule of $V_K^{\Phi}$. By Remark \ref{moduli inter of L} we know that $x_0 \in \mathcal Z(\mathbf v)$ if and only if $\mathbf v \subset \Phi L$, if and only if $\mathbf v \subset \Phi L \cap L$ (as $\mathbf v$ is $\Phi$-invariant). 
\end{rem}



\begin{defn}\label{def:minuscule}
    When $m=n$ and $T(\mathbf{v})$ is non-singular, we obtain a lattice $$L(\mathbf{v})=\mathbb{Z}_p v_1+\cdots \mathbb{Z}_p v_n\subseteq V_K^\Phi.$$ By the Cartan decomposition, $T(\mathbf v) \in \GL_n (\ZZ_p) \diag (p^{r_1}, p^{r_2},\cdots, p^{r_n}) \GL_n(\ZZ_p)$ for a unique non-increasing sequence of integers $r_1\ge \cdots\ge r_n$. Note that if we view the matrix $T(\mathbf v)^{-1}$ as a linear operator $V_K^{\Phi} \to V_K^{\Phi}$ using the basis $\mathbf v$, it sends $\mathbf v$ to the dual basis of $\mathbf v$, and in particular it sends any $\ZZ_p$-basis of $L(\mathbf v) $ to a $\ZZ_p$-basis of $ L(\mathbf v) ^{\vee} $. Therefore the tuple $(r_1,\cdots, r_n)$ is characterized by the condition that there is a basis $e_1,\ldots,e_n$ of $L(\mathbf{v})$ such that $p^{-r_1}e_1,\ldots, p^{-r_n}e_n$ form a basis of $L(\mathbf{v})^\vee$. From this characterization we also see that the tuple $(r_1,\cdots, r_n)$ is an invariant only depending on the lattice $L(\mathbf{v})$. We say $\mathbf{v}$ is \emph{minuscule} if $T(\mathbf v)$ is non-singular and $r_1=1, ~ r_n\ge0$. 
\end{defn}

\begin{rem}\label{rem:minsculespecialcyle}
	Suppose $m=n$ and $T(\mathbf v)$ is non-singular. Then $\mathbf v$ is minuscule if and only if $L(\mathbf{ v}) ^{\vee}$ is a vertex lattice. In this case by definition $\mathcal Z(\mathbf v) = \RZ_{L(\mathbf v) ^{\vee}}$. 
\end{rem}



\section{The intersection problem and the point-counting formula}\label{sec:intersection-problem}

\subsection{The GSpin Rapoport--Zink subspace}

From now on we assume $n\ge4$. Suppose the last basis vector $x_n\in V$ has norm 1. Then the quadratic subspace of dimension $n-1$ $$V^\flat=\mathbb{Z}_px_1+\cdots \mathbb{Z}_px_{n-1}$$ is also self-dual. Let $G^\flat=\GSpin(V^\flat)$. Analogously we define the element $$b^\flat=x_3(p^{-1}x_1+x_2)\in G^\flat(\mathbb{Q}_p)$$ and the cocharacter $$\mu^\flat: \mathbb{G}_m\rightarrow G^\flat,\quad t\mapsto t^{-1}x_1x_2+x_2x_1.$$ As in \S \ref{sec:gspin-rapoport-zink-spaces}, we have an associated GSpin Rapoport--Zink space $$\RZ^\flat=\RZ(G^\flat, b^\flat, \mu^\flat, C(V^\flat)).$$ The embedding $V^\flat\hookrightarrow V$ induces an embedding of Clifford algebras $C(V^\flat)\hookrightarrow C(V)$ and a closed embedding of group schemes $G^\flat\hookrightarrow G$ over $\mathbb{Z}_p$, which maps $b^\flat$ to $b$ and $\mu^\flat$ to $\mu$. Thus by the functoriality of Rapoport--Zink spaces (\cite[4.9.6]{Kim2013}), we have a closed immersion  $$\delta: \RZ^\flat\hookrightarrow \RZ$$ of formal schemes over $W$.

\subsection{Relation with the special divisor $\mathcal{Z}(x_n)$}

\label{sec:relat-with-spec}

For compatible choices of symplectic forms $\psi^\flat$ on $C(V^\flat)$ and $\psi$ on $C(V)$, the closed embedding of group schemes $\GSp(C(V^\flat),\psi^\flat)\hookrightarrow \GSp(C(V),\psi)$ induces a closed immersion of symplectic Rapoport--Zink spaces (\S \ref{sec:gspin-rapoport-zink-spaces}) $$\phi: \RZ(\mathbb{X}_0^\flat,\lambda_0^\flat)\hookrightarrow\RZ(\mathbb{X}_0, \lambda_0).$$ Since we have a decomposition of $\GSp(C(V^\flat),\psi^\flat)$-representations $$C(V)\cong C(V^\flat) \oplus C(V^\flat) x_n,$$ we know the moduli interpretation of $\phi$ is given by sending a triple  $(X^\flat,\rho^\flat,\lambda^\flat)$ to the $p$-divisible group $X=X^\flat \oplus X^\flat$ with the quasi-isogeny $\rho=\rho^\flat \oplus \rho^\flat$ and polarization $\lambda=\lambda^\flat \oplus \lambda^\flat$.

By the functoriality of Rapoport--Zink spaces (\cite[4.9.6]{Kim2013}), we have a commutative diagram of closed immersions
\begin{equation}
  \label{eq:funcotrialityRZ}
  \xymatrix{\RZ^\flat  \ar@{^(->}[r]^{\delta} \ar@{^(->}[d] &  \RZ \ar@{^(->}[d]\\ \RZ(\mathbb{X}_0^\flat,\lambda_0^\flat) \ar@{^(->}[r]^{\phi}& \RZ(\mathbb{X}_0, \lambda_0).}
\end{equation}
Here the two vertical arrows are induced by the closed immersions $\GSpin(V^\flat)\hookrightarrow \GSp(C(V^\flat),\psi^\flat)$ and  $\GSpin(V)\hookrightarrow \GSp(C(V),\psi)$ (\S \ref{sec:gspin-rapoport-zink-spaces}).

\begin{lem}
   Diagram (\ref{eq:funcotrialityRZ}) is Cartesian, i.e., we have
 \begin{equation}
   \label{eq:cartesian}
   \delta(\RZ^\flat)=\phi(\RZ(\mathbb{X}_0^\flat,\lambda_0^\flat))\cap\RZ
 \end{equation}
 inside $\RZ(\mathbb{X}_0,\lambda_0)$.
\end{lem}

\begin{proof}
By flat descent, to show that the closed formal subschemes on the two sides of \eqref{eq:cartesian} agree, it suffices to show that they have the same $k$-points and the same formal completion at every $k$-point (cf. \cite[5.2.7]{Bueltel2017}). The claim then follows from the observation that both the $k$-points and the formal completions have purely group theoretic description. 

In fact, the $k$-points of $\RZ^\flat=\RZ_{G^\flat}$, $\RZ(\mathbb{X}_0^\flat,\lambda_0^\flat)=\RZ_H$ and $\RZ=\RZ_G$ have the group theoretic description as the affine Deligne--Lusztig sets \eqref{eq:afdl} associated to the groups $G^\flat=\GSpin(V^\flat)$, $H=\GSp(C(V^\flat), \psi^\flat)$ and $G=\GSpin(V)$ respectively. Since $G^\flat=H\cap G$ inside $\GL(C(V))$, we know that both sides of \eqref{eq:cartesian} have the same $k$-points. Fix a $k$-point $x\in\RZ^\flat(k)$, then by \cite[3.2.12]{Howard2015}, $\widehat\RZ_{G^\flat,x}$ can be identified with $U_{G^\flat}^{\mu_x,\wedge}$, where $\mu_x: \mathbb{G}_{m,W}\rightarrow G^\flat_W$ gives a filtration that lifts the Hodge filtration for $x$, $U_{G^\flat}^{\mu_x}\subseteq G^\flat$ is the unipotent radical of the opposite parabolic group defined by $\mu_x$ (\cite[3.1.6]{Howard2015}) and  $U_{G^\flat}^{\mu_x,\wedge}$ is its formal completion along its identity section over $W$. Similarly, we can identify $\widehat{\RZ}_{H,x}$ and $\widehat{\RZ}_{G,x}$ as $U_{H}^{\mu_x,\wedge}$ and $U_{G}^{\mu_x,\wedge}$. Again because $G^\flat=H\cap G$, we know that the formal completions at $x$ of both sides of \eqref{eq:cartesian} agree inside $U_{\GL(C(V))}^{\mu_x,\wedge}$.
\end{proof}

\begin{lem}\label{lem:specialdivisor}
  $\delta(\RZ^\flat)=\mathcal{Z}(x_n)$.
\end{lem}

\begin{proof}
 Let $X^\flat$ be the universal $p$-divisible group over $\RZ^\flat$ and $\rho^\flat$ be the universal quasi-isogeny. Then it follows from the commutative diagram (\ref{eq:funcotrialityRZ}) that the image of $(X^\flat,\rho)$  under $\delta$ is given by the $p$-divisible group $(X^\flat\oplus  X^\flat,\rho^\flat \oplus \rho^\flat)$. Since $x_n$ has norm 1, right multiplication by $x_n$ swaps the two factors $C(V^\flat)$ and $C(V^\flat)x_n$.  It follows that the quasi-endomorphism $$(\rho^{\flat}\oplus \rho^{\flat} )\circ x_n\circ (\rho^{\flat}\oplus \rho^{\flat} )^{-1}: (X^\flat\oplus X^\flat)\to(X^\flat\oplus X^\flat)$$ (uniquely determined by the rigidity of quasi-isogenies) simply swaps the two factors, which is an actual endomorphism (i.e., swapping) of $X^{\flat} \oplus X^{\flat}$. By Definition \ref{def:specialcycle} of $\mathcal{Z}(x_n)$, we have $\delta(\RZ^\flat)\subseteq \mathcal{Z}(x_n)$.

Conversely, over $\mathcal{Z}(x_n)$ the universal $p$-divisible group $X$ admits an action of $C(x_n)^\mathrm{op} \otimes C(V)$, where $C(x_n)$ is the Clifford algebra of the rank one quadratic space $\mathbb{Z}_p x_n$. Notice $$C(x_n)^\mathrm{op} \otimes C(V)\cong (C(x_n)^\mathrm{op} \otimes C(x_n)) \oplus (C(x_n)^\mathrm{op} \otimes C(V^\flat)).$$ It follows that over $\mathcal{Z}(x_n)$ the universal $p$-divisible group $X$ admits an action of $C(x_n)^\mathrm{op} \otimes C(x_n)$, which is isomorphic to the matrix algebra $M_2(\mathbb{Z}_p)$. The two natural idempotents of $M_2(\mathbb{Z}_p)$ then decomposes $X$ as a direct sum of the form $X^\flat \oplus X^\flat$. Hence $\mathcal{Z}(x_n)\subseteq \phi(\RZ(\mathbb{X}_0^\flat,\lambda_0^\flat))\cap \RZ$. The latter is equal to $\delta(\RZ^\flat)$ by (\ref{eq:cartesian}) and hence $\mathcal{Z}(x_n)\subseteq \delta(\RZ^\flat)$.
\end{proof}

\begin{rem}
In the following we will only use the inclusion $\delta(\RZ^\flat)\subseteq\mathcal{Z}(x_n)$.
\end{rem}

\subsection{Arithmetic intersection of GGP cycles}
\label{sec:goal}
\begin{defn}
The closed immersion $\delta$ induces a closed immersion of formal schemes $$(\id, \delta): \RZ^\flat\rightarrow \RZ^\flat \times_W \RZ.$$ Denote by $\Delta$ the image of $(\id,\delta)$, which we call the \emph{GGP cycle}.  
\end{defn}

The embedding $V^\flat\hookrightarrow V$ also induces an embedding of quadratic spaces $V^{\flat,\Phi}_K\hookrightarrow V_K^\Phi$ and hence we can view $$J_{b^\flat}=\GSpin(V^{\flat,\Phi}_K)\hookrightarrow J_b$$ as an algebraic subgroup over $\mathbb{Q}_p$.

For any $g\in J_b(\mathbb{Q}_p)$, we obtain a formal subscheme $$g\Delta:=(\id\times g)\Delta\subseteq \RZ^\flat \times_W \RZ,$$ via the action of $g$ on $\RZ$. Our goal is to compute the arithmetic intersection number $$\langle \Delta,g\Delta\rangle,$$ when $g$ is regular semisimple and minuscule. 

\begin{defn}\label{def:relative regular}
  We say $g\in J_b(\mathbb{Q}_p)$ is \emph{regular semisimple} if the  $\mathbf{v}(g):=(x_n, gx_n,\ldots, g^{n-1}x_n)$ forms a $\mathbb{Q}_p$-basis of $V_K^\Phi$. Equivalently, the fundamental matrix $T(g):=T(\mathbf{v}(g))$ is non-singular (Definition \ref{def:specialcycle}).\ignore{ Equivalently, the stabilizer of $g$, for the conjugation action of the subgroup $J_{b^\flat}$, lies in the center ($\cong\mathbb{G}_m$) of $J_{b^\flat}$.}
 We say $g$ is \emph{minuscule} if $\mathbf{v}(g)$ is minuscule (Definition \ref{def:minuscule}).
\end{defn}



\subsection{Fixed points}

Let $g\in J_b(\QQ_p)$ and let $\RZ^g\subseteq \RZ$ be the fixed locus of $g$. Then by definition we have $$\Delta\cap g\Delta\cong\delta(\RZ^\flat)\cap \RZ^g.$$

\begin{defn}
Let $g\in J_b(\mathbb{Q}_p)$ be regular semisimple. We define the lattice $$L(g):=\mathbb{Z}_p x_n+\cdots \mathbb{Z}_p g^{n-1}x_n\subseteq V_K^\Phi.$$  
\end{defn}

\begin{lem}\label{lem:Lg} Inside $\RZ$ both the formal subschemes $\RZ^g $ and $\delta(\RZ^{\flat})$ are stable under $p^{\ZZ}$. Moreover, under the bijection (\ref{eq:RZk}), we have
  \begin{enumerate}
  \item   $p^\mathbb{Z}\backslash\delta(\RZ^\flat(k))\cong\{L=L^\flat \oplus W x_n: L^\flat\subseteq V^\flat_K \text{ special lattices}\}$.
  \item $p^\mathbb{Z}\backslash\delta(\RZ^\flat(k))\cong\{L\text{ special lattices}: x_n \in L \}$.
  \item $p^\mathbb{Z}\backslash\RZ^g(k)\cong\{L \text{ special lattices}: gL=L\}$.
  \item  $p^\mathbb{Z}\backslash(\delta(\RZ^\flat(k))\cap\RZ^g(k))\cong \{L \text{ special lattices}: gL=L, L\supseteq L(g)_W\}$.
  \end{enumerate}
\end{lem}

\begin{proof}
	Since $p^{\ZZ}$ is central in $J_b(\QQ_p)$, we know $\RZ^g$ is stable under $p^{\ZZ}$. The morphism $\delta: \RZ^{\flat} \to \RZ$ is equivariant with respect to the natural inclusion $J_{b^{\flat}} (\QQ_p) \to J_b (\QQ_p)$, and the morphism $J_{b^{\flat}} \to J_b$ restricts to the identity between the centers $\GG_m$ of $J_{b^{\flat}}$ and of $J_b$. It follows that $\delta$ is equivariant for the $p^{\ZZ}$ action, and so $\delta(\RZ^{\flat})$ is stable under $p^{\ZZ}$. We now prove the statements (1) to (4).   
  \begin{enumerate}
  \item For a point $L^\flat\in p^\mathbb{Z}\backslash\RZ^\flat(k)$, we can write $L^\flat=h^\flat\mu^\flat(p^{-1}).V^\flat_W\subseteq V^\flat_K$, for some $h^\flat\in G^\flat(K)$. Then its image under $\delta$ is given by $L=h\mu(p^{-1}).V_W\subseteq V_K$, where $h$ is the image of $h^\flat$ in $G(K)$. By $V=V^\flat \oplus \mathbb{Z}_p x_n$ and the compatibility between $h,\mu$ and $h^\flat, \mu^\flat$, we know that $L=L^\flat \oplus W x_n$. 
  \item Suppose $L$ is a special lattice with $x_n\in L$. Since $x_n$ has norm 1, we know that $L=L' \oplus Wx_n$ is the direct sum of $Wx_n$ and its orthogonal complement $L'$ in $L$. One can check $L'\subseteq V_K^\flat$ is also a special lattice. This finishes the proof in view of item (1). 
  \item This is clear since $\RZ^g(k)$ is the fixed locus of $g$.
  \item For a point $L\in p^\mathbb{Z}\backslash(\delta(\RZ^\flat(k))\cap\RZ^g(k))$, by items (1) (3), we have $L=L^\flat \oplus W x_n$ and $gL=L$. It follows from $x_n\in L$ that $gx_n,\ldots, g^{n-1}x_n\in L$, and so $L\supseteq L(g)_W$. Conversely, if a point $L\in \RZ(k)$ satisfies $gL=L$ and $L \supset L(g)_W$, then $  L \in p^\mathbb{Z}\backslash(\delta(\RZ^\flat(k))\cap\RZ^g(k)) $ by items (2) and (3)\qedhere
  \end{enumerate}
\end{proof}

\begin{defn}
We say a vertex lattice $\Lambda$ is  a \emph{$g$-vertex lattice} if $g\Lambda=\Lambda$ and $\Lambda\subseteq L(g)^\vee$. Denote the set of all $g$-vertex lattices by $\VL(g)$. In general, if a vertex lattice $\Lambda$ satisfies $g\Lambda=\Lambda$, then $g$ induces an action on $\Omega_0=\Lambda/\Lambda^\vee$, which further induces an action $\bar g$ on $\RZ_\Lambda^\mathrm{red}$ and $\BT_\Lambda$. We denote the fixed locus of $\bar g$ on $\BT_\Lambda$ by $\BT_\Lambda^{\bar g}$. 
\end{defn}

\begin{prop}\label{prop:fixed-points}
  $$p^\mathbb{Z}\backslash(\delta(\RZ^\flat)\cap\RZ^g)(k)=\coprod_{\Lambda\in \VL(g)}p^\mathbb{Z}\backslash\BT_\Lambda^{\bar g} (k).$$
\end{prop}

\begin{proof}
By Lemma \ref{lem:Lg},  it suffices to show the $k$-points of the right hand side are in bijection with special lattices $L$ such that $gL=L$ and $L\supseteq L(g)_W$. Notice that any special lattice $L$ is self-dual, so the condition $L\supseteq L(g)_W$ is equivalent to the condition $L\subseteq L(g)_W^\vee$. Since $\Lambda(L)_W$ is the minimal $\Phi$-invariant lattice containing $L$ (\S \ref{sec:spec-latt-vert}), and $L(g)_W^\vee$ is $\Phi$-invariant, we know that the condition $L\subseteq L(g)_W^\vee$ is equivalent to the condition $\Lambda(L)\subseteq L(g)^\vee$. The result now follows from taking $\bar g$-invariants and $g$-invariants of the two sides of the bijection (\ref{eq:bt}).
\end{proof}

\subsection{Fixed points in a Bruhat--Tits stratum}

Let $\Lambda$ be a vertex lattice and $\Omega_0=\Lambda/\Lambda^\vee$ (\S \ref{sec:vertex-lattices}). By the isomorphism (\ref{eq:btdl}), $p^\mathbb{Z}\backslash\BT_\Lambda$ is disjoint union of two isomorphic Deligne--Lusztig varieties $X_B(w^{\pm})$ associated to the Coxeter elements $w^\pm$ for $\SO(\Omega_0)$. Write $X:=X_B(w^\pm)$. To compute $p^\mathbb{Z}\backslash\BT_\Lambda^{\bar g}$, it suffices to compute the $\bar g$-fixed points $X^{\bar g}$.

\begin{defn}
We say a semisimple element $\bar g\in \SO(\Omega_0)$ is \emph{regular} 
if $Z^\circ(\bar g)$, the identity component of the centralizer of $\bar g$ in $\SO(\Omega_0)$, is a (necessarily maximal) torus\footnote{Note the difference with Definition \ref{def:relative regular}. The conflict of the usage of the word "regular" should hopefully not cause confusion.}. 
\end{defn}

\begin{prop}\label{prop:nonempty}
  Let $\Lambda$ be a vertex lattice and let $\bar g\in \SO(\Omega_0)(\mathbb{F}_p)$.
  \begin{enumerate}
  \item \label{item:1} $X^{\bar g}$ is non-empty if and only if $\bar g$ is semisimple and contained in a maximal torus of Coxeter type.
  \item \label{item:2} $X^{\bar g}$ is non-empty and finite if and only if $\bar g$ is regular semisimple and contained in a maximal torus of Coxeter type. In this case, the cardinality of $X^{\bar g}$ is given by $t_\Lambda/2$.
  \end{enumerate}
\end{prop}

\begin{rem}\label{rem:coxeter}
Recall that a maximal torus $T'$ is of Coxeter type if $T'=h T h^{-1}$ for some $h\in \SO(\Omega_0)$ such that $h^{-1}\Phi(h)$ lifts to a  Coxeter element $w$ in the Weyl group $W(T)=N(T)/T$. In other words, $T'$ is conjugate to $T$ over $k$ but its Frobenius structure is given by $w\cdot\Phi$. For the Coxeter element $w=w^\pm$ constructed in \S \ref{sec:deligne-luszt-vari}, we know that an element $(\lambda_1,\ldots,\lambda_d,\lambda_d^{-1},\ldots,\lambda_1^{-1})$ of $T(k)$ is fixed by $w\cdot\Phi$ if and only if $$(\lambda_1,\lambda_2,\ldots,\lambda_{d-1}, \lambda_d)=(\lambda_d^{\mp p},\lambda_1^p \ldots, \lambda_{d-2}^p,\lambda_{d-1}^{\pm p}).$$  It follows that a semisimple element $\bar g\in \SO(\Omega_0)(\mathbb{F}_p)$ is contained in a maximal torus of Coxeter type if and only if the eigenvalues of $\bar g$ on $\Omega_0 \otimes k$ belong to a single Galois orbit.
\end{rem}


\begin{proof}
\begin{enumerate}
 \item  Suppose $X^{\bar g}$ is non-empty. Then it is a general fact about Deligne--Lusztig varieties that $\bar g$ must be semisimple (\cite[5.9 (a)]{Lusztig2011}). Let $T(w)\subseteq \SO(\Omega_0)$ be a torus of Coxeter type (associated to $w=w^+$ or $w^-$) and $B(w)\supseteq T(w)$ be a Borel. Assume $\bar g$ is semisimple. Then we know from \cite[Proposition 4.7]{Deligne1976} that $X^{\bar g}$ is a disjoint union of Deligne-Lusztig varieties $X_{T'\subseteq B'}$ for the group $G'=Z^\circ(\bar g)$ and the pairs $$(T', B')=(h T(w) h^{-1}, h B(w) h^{-1}\cap G'),$$ where $h$ runs over classes $G'(\mathbb{F}_p)\backslash\SO(\Omega_0)(\mathbb{F}_p)$ such that $\bar g\in h T(w)h^{-1}$. Therefore $X^{\bar g}$ is non-empty if and only if there exists $h\in \SO(\Omega_0)(\mathbb{F}_p)$ such that $\bar g\in h T(w) h^{-1}$, if and only if $\bar g$ is contained in a maximal torus of Coxeter type (as so is $T(w)$).
 \item By part (1) we know that $X^{\bar g}$ is further finite if and only if all $X_{T'\subseteq B'}$ are zero dimensional, if and only if all $B'=hBh^{-1}\cap G'$ are tori. This happens exactly when $G'=Z^\circ(\bar g)$ itself is a torus, i.e., when $\bar g$ is regular. In this case, $G'$ is a maximal torus of Coxeter type in $\SO(\Omega_0)$ and the cardinality of $X^{\bar g}$ is equal to the cardinality of $N(T(w))(\mathbb{F}_p)/T(w)(\mathbb{F}_p)$. The latter group is isomorphic to $(N(T(w))/T(w))^\Phi$ by Lang's theorem and hence is isomorphic to the $\Phi$-twisted centralizer of $w$ in the Weyl group $W(T)=N(T)/T$: $$Z_\Phi(w):=\{x\in W(T): xw=w\Phi(x)\}.$$ The cardinality of $Z_\Phi(w)$ is known as the Coxeter number of the group $\SO(\Omega_0)$, which is equal to $d=t_\Lambda/2$ since $\SO(\Omega_0)$ is a non-split even orthogonal group (\cite[1.15]{Lusztig1976/77}).  \qedhere
\end{enumerate}
\end{proof}

\subsection{Point-counting in the minuscule case} Let $g\in J_b(\mathbb{Q}_p)$ be regular semisimple and minuscule. Then $\Omega_0(g):=L(g)^\vee/L(g)$ is a $\mathbb{F}_p$-vector space (see Definition \ref{def:minuscule}), and hence $L(g)^\vee$ is a vertex lattice.

\begin{rem}\label{rem:regular}
If $\RZ^g$ is non-empty, then $g$ fixes some vertex lattice and so we know that the characteristic polynomial of $g$ has $\ZZ_p$-coefficients. It follows that $L(g)$ is a $g$-stable lattice, from which it also follows easily that $L(g)^{\vee}$ is $g$-stable. Hence by definition $L(g)^\vee$ is a $g$-vertex lattice. The induced action of $g$ on $\Omega_0(g)$, denoted by $\bar g\in \SO(\Omega_0(g))(\mathbb{F}_p)$, makes $\Omega_0(g)$ a $\bar g$-cyclic $\mathbb{F}_p$-vector space. It follows that the minimal polynomial of $\bar g$ is equal to its characteristic polynomial.
\end{rem}

From now on we assume $\RZ^g$ is non-empty. Let $\bar g \in \SO(\Omega_0(g))(\mathbb{F}_p)$ be as in Remark \ref{rem:regular}.


\begin{defn}
For any polynomial $R(T)$, we define its \emph{reciprocal} to be $$R^*(T):=T^{\deg R(T)}\cdot R(1/T).$$ We say $R(T)$ is \emph{self-reciprocal} if $R(T)=R^*(T)$.
\end{defn}

\begin{defn}
 Let $P(T)\in \mathbb{F}_p[T]$ be the characteristic polynomial of $\bar g\in \SO(\Omega_0(g))$. Then $P(T)$ is self-reciprocal. For any monic irreducible factor $Q(T)$ of $P(T)$, we denote by $m(Q(T))$ to be the multiplicity of $Q(T)$ appearing in $P(T)$.    
\end{defn}

\begin{thm}\label{thm:pointscounting}Assume $\RZ^g $ is non-empty. Then 
$p^\mathbb{Z}\backslash(\delta(\RZ^\flat)\cap\RZ^g)(k)$ is non-empty if and only if $P(T)$ has a unique self-reciprocal monic irreducible factor $Q(T)$ such that $m(Q(T))$ is odd. In this case, $p^\mathbb{Z}\backslash(\delta(\RZ^\flat)\cap\RZ^g)(k)$ is finite and has cardinality $$\deg Q(T)\cdot \prod_{R(T)}(1+m(R(T))),$$ where $R(T)$ runs over all non-self-reciprocal monic irreducible factors of $P(T)$.
\end{thm}

\begin{proof}
  By Proposition \ref{prop:fixed-points}, we know that $p^\mathbb{Z}\backslash(\delta(\RZ^\flat)\cap\RZ^g)(k)$ is non-empty if and only if $p^\mathbb{Z}\backslash\BT_\Lambda^g$ is non-empty for some $\Lambda\in \VL(g)$. For any $\Lambda\in \VL(g)$, by definition we have a chain of inclusions of lattices $$L(g)\subseteq \Lambda^\vee\subseteq \Lambda\subseteq L(g)^\vee,$$ which induces a filtration of $\mathbb{F}_p$-vector spaces, $$0\subseteq \Lambda^\vee/L(g)\subseteq \Lambda/L(g)\subseteq \Omega_0(g).$$ It follows that the map $\Lambda\mapsto \Lambda^\vee/L(g)$ gives a bijection
\begin{equation}
  \label{eq:VLg}
  \VL(g)\cong\{\text{totally isotropic }\bar g\text{-invariant subspaces }U\subseteq\Omega_0(g)\}.
\end{equation} By the bijection (\ref{eq:VLg}), $\VL(g)$ is non-empty if and only if there is a totally isotropic $\bar g$-invariant subspace $U$ of $\Omega_0(g)$. Such a subspace $U$ induces a filtration
  \begin{equation}
    \label{eq:filtration}
    0\subseteq U\subseteq U^\perp\subseteq \Omega_0(g).
  \end{equation}
 Since $U$ and $U^\perp$ are $\bar g$-invariant, we obtain a decomposition of the characteristic polynomial
  \begin{equation}
    \label{eq:decompoPT}
    P(T)=P_1(T)Q(T)P_2(T)
  \end{equation}
  where $P_1(T), Q(T), P_2(T)$ are respectively the characteristic polynomials of $\bar g$ acting on the associated graded $U$, $U^\perp/U$ and $\Omega_0(g)/U^\perp$. Notice the non-degenerate quadratic form on $\Omega_0(g)$ identifies $\Omega_0(g)/U^\perp$ with the linear dual of $U$, from which we know that $P_2(T)=P_1^*(T)$. Similarly, we know that $Q(T)=Q^*(T)$, i.e., $Q(T)$ is self-reciprocal.

  Let $\Lambda=L(g)+U^\perp$ be the $g$-vertex lattice corresponding to $U$ under the bijection (\ref{eq:VLg}) and let $\Omega_0=\Lambda/\Lambda^\vee$ and $\bar g_0\in \SO(\Omega_0)(\mathbb{F}_p)$ be the induced action of $\bar g$ on $\Omega_0$. By Remark \ref{rem:regular}, the minimal polynomial of $\bar g$ is equal to its characteristic polynomial $P(T)$. Thus the minimal polynomial of $\bar g_0$ is equal to its characteristic polynomial $Q(T)$  under the decomposition (\ref{eq:decompoPT}). If $\bar g_0$ is semisimple, then its eigenvalues are distinct. If $\bar g$ is further contained in a torus of Coxeter type, then we know that its eigenvalues belong to a single Galois orbit (Remark \ref{rem:coxeter}), so $Q(T)$ is irreducible. Conversely, if $Q(T)$ is irreducible, then clearly $\bar g_0$ is semisimple and contained in a torus of Coxeter type. Hence we know that $\bar g_0$ is semisimple and contained in a torus of Coxeter type if and only if $Q(T)$ is irreducible.

  Therefore by  Proposition \ref{prop:nonempty} (\ref{item:1}), $\BT_\Lambda^{\bar g_0}$ is non-empty if and only if $Q(T)$ is irreducible.  In this case, $\bar g_0$ is indeed regular semisimple and the cardinality of $p^\mathbb{Z}\backslash\BT_\Lambda^{\bar g_0}$ is equal $2\cdot\#X^{\bar g_0}$ (due to two connected components), which is equal to $\dim_{\mathbb{F}_p} \Omega_0=\deg Q(T)$ by Proposition \ref{prop:nonempty} (\ref{item:2}).

  Since $P_2(T)=P_1^*(T)$, we know the multiplicity of $R(T)$ in $P_1(T)P_2(T)$ is even for any self-reciprocal factor $R(T)$. Hence $Q(T)$ is the unique self-reciprocal monic irreducible factor of $P(T)$ such that $m(Q(T))$ is odd. Finally, the factorizations  (\ref{eq:decompoPT}) with $P_2(T)=P_1^*(T)$ corresponds bijectively to the filtrations (\ref{eq:filtration}). The proof is now finished by noticing that the number of such factorizations is exactly given by  $$\prod_{R(T)}(1+m(R(T))),$$  where $R(T)$ runs over all monic irreducible factors of $P(T)$ such that $R(T)\ne R^*(T)$.
\end{proof}

\section{The reducedness of minuscule special cycles}\label{sec:reduc-minusc-spec}

\subsection{The analogue of a result of Madapusi Pera on special cycles}

\begin{defn}\label{defn:isotropic line}
	Let $\oo$ be an arbitrary $\ZZ[1/2]$-algebra. Assume $\oo$ is local. Let $\mathbf L$ be a finite free $\oo$-module equipped with the structure of a self-dual quadratic space over $\oo$. By an \textit{isotropic line} in $\mathbf L$ we mean a direct summand of rank one on which the quadratic form is zero.  
\end{defn}
We start with a general lemma on Clifford algebras.
\begin{lem}\label{lem:Cliff}
	Let $\oo$ and $\mathbf L$ be as in Definition \ref{defn:isotropic line}. Let $C(\mathbf L)$ be the associated Clifford algebra. Let $\xi \in \mathbf L$ be an $\oo$-generator of an isotropic line. Let $\ker (\xi)$ be the kernel of the endomorphism of $C(\mathbf L)$ given by left multiplication by $\xi$. Then for any $v\in \mathbf L$, left multiplication by $v$ preserves $\ker(\xi)$ if and only if $v$ is orthogonal to $\xi$.  
\end{lem}
\begin{proof}
	Assume $v $ is orthogonal to $\xi$. Then $v\xi  = -\xi v$, so $v$ preserves  $\ker (\xi)$. 
	
	Conversely, assume $v$ preserves $\ker (\xi)$. Write $q$ for the quadratic form and $\lprod{,}$ the corresponding bilinear pairing. Since $\oo \xi$ is a direct summand of $\mathbf L$, there exists an $\oo$-module homomorphism $\mathbf L \to \oo$ sending $\xi$ to $1$. Since $\mathbf L $ is self-dual, we know that there exists $\zeta \in \mathbf L$ representing such a homomorphism. Namely we have $$\lprod{\zeta, \xi} = 1. $$ 
	It immediately follows that we have an $\oo$-module direct sum $\mathbf L = \xi^{\perp} \oplus \oo \zeta$. Replacing $\zeta$ by $\zeta - \frac{q(\zeta)}{2} \xi$, we may arrange that $\zeta$ is isotropic.  We have 
	$$q(\zeta +\xi) = 2 \lprod{\zeta, \xi} = 2,$$ and in $C(\mathbf L)$ we have 
	$$ q(\zeta +\xi) = \zeta \xi + \xi \zeta. $$ Hence in $C(\mathbf L)$ we have \begin{align}\label{eq: =2}
	\xi \zeta + \zeta \xi =2
	\end{align}Write $$v = v_1 + \lambda \zeta,$$ with $v_1 \in \xi^{\perp}$ and $\lambda \in \oo$. By the first part of the proof we know that $v_1$ preserves $\ker (\xi)$. Therefore $\lambda \zeta$ preserves $\ker (\xi)$. Note that $\xi \in \ker (\xi)$ as $\xi$ is isotropic. It follows that, in $C(\mathbf L)$, 
	$$ 0 \xlongequal {\lambda\zeta \mbox{ preserves }\ker (\xi) }\xi (\lambda \zeta) \xi \xlongequal{(\ref{eq: =2})} \lambda (2-\zeta \xi) \xi \xlongequal{\xi \mbox{ isotropic}} 2\lambda \xi. $$ This is possible only when $\lambda =0$, and hence we have $v= v_1 \in \xi^{\perp}$.
\end{proof}

The next result is a Rapoport--Zink space analogue of \cite[Proposition 5.16]{MPspin} which is in the context of special cycles on $\Gspin$ Shimura varieties. We only state a weaker analogue as it is sufficient for our need. The proof builds on loc. cit. too. We first introduce some definitions.

\begin{defn}\label{defn of Fil^1}
Denote by $y_{00}$ the distinguished $k$-point of $\RZ$ corresponding to $\mathbb X_0$ and the identity quasi-isogeny. Let $y_0 \in \RZ (k)$ be an arbitrary element. Let $L$ be the special lattice corresponding to $y_0$ under (\ref{eq:RZk}).  When $y_0 = y_{00}$, we have $\Phi L = V_W$ (cf. the discussion below (\ref{eq:RZk})). In this case define $\Fil ^1  (\Phi L) _k$ to be the one-dimensional subspace of $V_k$ defined by the cocharacter $\mu $ of $G_W$ and the representation $G_k \to \GL(V_k)$. For general $y_0$, let $g \in X_{G, b ,\mu^{\sigma}} (k)$ be associated to $y_0$. Then $\Phi L = g V_W$ and $g$ induces a map $V_k \to (\Phi L)_k$ (cf. loc. cit.).  Define $\Fil^1 (\Phi L) _k$ to be the image of $\Fil ^1 V_k$ under the last map. 
\end{defn}
\begin{rem}
By our explicit choice of $\mu$ in \S \ref{explicit choices}, the submodule $\ZZ_p x_2$ in $V$ is of weight $1$ with respect to $\mu$, and $\oplus_{1\leq i \leq n, i\neq 2}\ZZ_px _i$ is of weight $0$ with respect to $\mu$, so $\Fil^1 V_k = k x_2$.
\end{rem}
\begin{rem}\label{rem: not used}
	$\Fil ^1 (\Phi L)_k$ is the orthogonal complement in $(\Phi L) _k$ of $(L \cap \Phi L) _k$. However we will not need this description in the sequel. 
\end{rem}

\begin{defn}
Let $\mathscr C$ be the category defined as follows:
\begin{itemize}
	\item Objects in $\mathscr C$ are triples $(\oo, \oo\to k, \delta)$, where $\oo$ is a local artinian $W$-algebra, $\oo \to k$ is a $W$-algebra map, and $\delta$ is a nilpotent divided power structure on $\ker (\oo \to k)$. 
	\item Morphisms in $\mathscr C$ are $W$-algebra maps that are compatible with the structure maps to $k$ and the divided power structures.  
\end{itemize}
In the following we will abuse notation to write $\oo$ for an object in $\mathscr C$.
\end{defn}

Let $y_0 \in \RZ(k)$ be an arbitrary element. Let $\mathbf v$ be as in Definition \ref{def:specialcycle} such that the special cycle $\mathcal Z: = \mathcal Z (\mathbf v)$ contains $y_0$. In particular $\mathbf v \subset L\cap \Phi L$ by Remark \ref{rem: v and L}. Let $\widehat{\RZ}_{y_0}$ and $\widehat{\mathcal Z}_{y_0}$ be the formal completions of $\RZ$ and $\mathcal Z$ at $y_0$ respectively. 
\begin{thm}\label{thm: MP}
For any $\oo \in \mathscr C$ there is a bijection  $$ f_{\oo}: \widehat{\RZ}_{y_0} (\oo)  \isom  \set{\mbox{isotropic lines in $(\Phi L)_{\oo}: = \Phi L \otimes_W \oo$ lifting $\Fil^1 (\Phi L) _k$}} $$ such that the following properties hold. Here we equip $(\Phi L)_{\oo}$ with the $\oo$-bilinear form obtained by extension of scalars of the $W$-bilinear form on $\Phi L$.
\begin{enumerate}
	\item  $f_{\oo}$ is functorial in $\oo \in \mathscr C$ in the following sense. Let $\oo' \in \mathscr C$ be another object of $\mathscr C$ and let $\phi : \oo \to \oo'$ be a morphism in $\mathscr C$. Then we have a commuting diagram. 
	$$\xymatrix{ \widehat{\RZ}_{y_0} (\oo) \ar[d]^{f_{\oo}} \ar[r] &  \widehat{\RZ}_{y_0} (\oo') \ar[d]^{f_{\oo'}}\\  \im (f_{\oo} ) \ar[r] &  \im (f_{\oo'})}$$ Here the top horizontal map is the natural map induced by $\phi$, and the bottom horizontal map is given by base change along $\phi$. 
	\item $f_{\oo}$ restricts to a bijection 
	 $$
	 f_{\oo, \mathbf v}:  \widehat{\mathcal Z}_{y_0} (\oo) \isom   $$$$ \set{\mbox{isotropic lines in $(\Phi L)_{\oo}$ lifting $\Fil^1 (\Phi L) _k$}  \mbox{ and orthogonal to the image of $\mathbf v$ in $(\Phi L)_{\oo}$} }
	 $$
\end{enumerate}

\end{thm}
\begin{proof}  
	
The existence and construction of the bijection $f_{\oo}$ and the property (1) are consequences of \cite[Proposition 5.16]{MPspin} and the global construction of $\RZ$ in \cite{Howard2015} using the integral model of the $\Gspin$ Shimura variety. We explain this more precisely below.

 Consider  $$\Shh = \Shh_{U_p U^p},$$ the canonical integral model over $\ZZ_{(p)}$ of the Shimura variety associated to the $\Gspin $ Shimura datum associated to a quadratic space $V_{\QQ}$ over $\QQ$, at a suitable level $U^p$ away from $p$ and a hyperspecial level $U_p$ at $p$. See \cite[\S 7]{Howard2015} or \cite{MPspin} for more details on this concept. By \cite[7.2.3]{Howard2015}, we may assume that the following package of data:
\begin{itemize}
	\item the Shimura datum associated to $V_{\QQ}$,
	\item the Kuga-Satake Hodge embedding (cf. \cite[4.14]{Howard2015}) of the Shimura datum into a $\Gsp$ Shimura datum,
	\item the chosen hyperspecial level at $p$,
	\item an element $x_{00} \in \varprojlim_{U^p} \Shh_{U^p U_p} (k)$,
\end{itemize} 
induces, in the fashion of \cite[3.1.4]{Howard2015}, the local unramified Shimura-Hodge datum that we used to define $\RZ$. Let $\widehat{\Shh}$ be the formal scheme over $\ZZ_p$ obtained from $p$-adic completion of $\Shh$, and let $\widehat{\Shh}_W$ be the base change to $W$ of $\widehat{\Shh}$. Then as in \cite[3.2.14]{Howard2015}, we have a morphism of formal schemes over $W$:
 $$\Theta: \RZ \to \widehat{\Shh }_W. $$	
 We know that $\Theta$ maps $y_{00}$ to the $k$-point of $\widehat{\Shh}_W$ induced by $x_{00}$. 
Moreover, let $$x_0: = \Theta(y_0) \in \widehat{\Shh} _W (k) = \Shh(k)$$ and let $\widehat{U}$ be the formal completion of $\Shh$ at $x_0$ (or, what amounts to the same thing, the formal completion of $ \widehat{\Shh} _W$ at $x_0$). By the construction of $\RZ$ in \cite[\S 3]{Howard2015}, we know that $\Theta$ induces an isomorphism $ \widehat{\RZ}_{y_0} \isom \widehat{U}$. 

In \cite{MPspin}, two crystals $\mathbf H_{\cris}, \mathbf L_{\cris}$ are constructed on $\Shh _{k}$. (In fact \cite{MPspin} works over $\FF_p$, but we always base change from $\FF_p$ to $k$.) Here $\mathbf H_{\cris}$ is by definition the first relative crystalline cohomology of the Kuga-Satake abelian scheme over $\Shh_k$ in the sense of loc. cit.\footnote{See footnote \ref{foot}.}The specialization of $\mathbf H_{\cris}$ over $\spec k$ via $x_{00}$ is identified with the Dieudonn\'e module $C_W$, which is the covariant Diedonn\'e module of the $p$-divisible group $\mathbb X_0$ considered in this article (and \cite{Howard2015}) and the contravariant Diedonn\'e module of the Kuga-Satake abelian variety at $x_{00}$ considered in \cite{MPspin}.\footnote{\label{foot}Due to different conventions, the Kuga-Satake abelian scheme (and $p$-divisible group) considered by Madapusi Pera in \cite{MPspin} is different from that considered by Howard-Pappas in \cite{Howard2015}. In fact they are dual to each other.}
 Moreover, the embedding $V \hookrightarrow \End_{\ZZ_p} (C)$ has a cristalline realization, which is a sub-crystal $\mathbf L_{\cris} $ of $\End (\mathbf{H}_{\cris})$. For details see \cite[\S 4]{MPspin}. Among others, $\mathbf L_{\cris}$ has the following structures: 
\begin{itemize}
	\item Its specialization $\mathbf L_{\cris, x_0}$ to any $x_0 \in \Shh (k)$, viewed as a $W$-module, has the structure of a $W$-quadratic space. 
	\item  $\mathbf L_{\cris, x_0}\otimes_W k$ contains a canonical isotropic line $\Fil ^1 (\mathbf L_{\cris, x_0}\otimes_W k)$. 
	\end{itemize}

By the definition of $\Theta$ and the definition of the parametrization of $\RZ(k)$ by the affine Deligne-Lusztig set (cf. \cite[\S 2.4]{Howard2015}), we know that when $y_0 \in \RZ (k)$ corresponds to the special lattice $L$ under (\ref{eq:RZk}), the following statements are true: 
\begin{enumerate}[label=(\alph*)]
	\item\label{item:a} There is an isomorphism of Dieudonn\'e modules $(gC)_W \isom \mathbf H_{\cris, x_0}$.
	\item\label{item:b} There is a $W$-linear isometry $(\Phi L)_W \isom \mathbf L_{\cris,x_0} $ under which $\Fil ^1 (\Phi L)_k$ is identified with $$\Fil ^1 (\mathbf L_{\cris, x_0}\otimes_W k).$$ 
	\item We have a commutative diagram:
	$$ \xymatrix{ \Phi L  \ar@{^{(}->}[r]  \ar[d] &  \End_W ((gC)_W) \ar[d] \\		 \mathbf L_{\cris, x_0} \ar@{^{(}->}[r]  & \End_W(\mathbf H_{\cris, x_0}),}$$ where 
        \begin{itemize}
 	\item the right vertical map is induced by the map in \ref{item:a}.          
 	\item the left vertical map is the map in \ref{item:b}. 
 	\item the bottom horizontal map arises from the fact that $\mathbf L_{\cris}$ is a sub-crystal of $\End(\mathbf H_{\cris})$.
 \end{itemize}
\end{enumerate}
In the rest of the proof we make the identifications in \ref{item:a} and \ref{item:b} above and omit them from the notation. Abbreviate $\mathbf H: = \mathbf H_{\cris ,x_0}$ and $\mathbf L := \mathbf L_{\cris, x_0}$.

Now in \cite[Proposition 5.16]{MPspin} Madapusi Pera constructs a bijection 
$$ \widehat{U}(\oo)  \isom  \set{\mbox{isotropic lines in $\mathbf{L}_{\oo}: = \mathbf L\otimes_W \oo$ lifting $\Fil^1 \mathbf L _k$}} .$$
Moreover by the construction given in loc. cit. the above bijection is functorial in $\oo\in \mathscr C$. We define $f_{\oo}$ as the above bijection precomposed with the isomorphism $\Theta: \widehat{\RZ}_{y_0} \isom \widehat{ U}$. 
 
 It remains to prove property (2). Note that $\mathbf H = gC_W$ is the \emph{covariant} Dieudonn\'e module of the $p$-divisible group $X_{y_0}$ over $k$ determined by $y_0\in \RZ(k)$. Given $y \in \widehat{\RZ} _{y_0} (\oo)$ lifting $y_0$, by Grothendieck-Messing theory (for covariant Dieudonn\'e modules) we know that $y \in \widehat {\mathcal Z}_{y_0}$ if and only if the image of $\mathbf v$ in $\End_{\oo} (\mathbf H_{\oo})$ stabilizes $\Fil^1 \mathbf H_{\oo } \subset \mathbf H_{\oo}$, where $\Fil^1 \mathbf H_{\oo } $ is the Hodge filtration corresponding to the deformation from $k$ to $\oo$ of the $X_{y_0}$ determined by $y$. Now, as is stated in the proof of \cite[Proposition 5.16]{MPspin}\footnote{Madapusi Pera defines $\Fil^1 \mathbf H_{\oo}$ using the contravariant Grothendieck-Messing theory of the $p$-divisible group of the Kuga-Satake abelian scheme in his sense, which is the same as the covariant Grothendieck-Messing theory of the $p$-divisible group over $\widehat{U}$ transported via $\Theta$ from the universal $p$-divisible group over $\RZ$ in the sense of Howard-Pappas.}, we know that $\Fil^1 \mathbf H_{\oo}$ is the kernel in $\mathbf H_{\oo}$ of any $\oo$-generator $\xi$ of the isotropic line $f_{\oo}(y)$. Here $\xi\in \mathbf L_{\oo}$ is viewed as an element of $\End _{\oo}(\mathbf H_{\oo})$. By Lemma \ref{lem:Cliff}, $\mathbf v$ preserves $\Fil^1 \mathbf H_{\oo} = \ker \xi$ if and only if $\mathbf v $ is orthogonal to $\xi$ (inside $\mathbf L_{\oo}$). Thus $y \in \widehat {\mathcal Z}_{y_0}$ if and only if $f_{\oo} (y)$ is orthogonal to the image of $\mathbf v$ in $\mathbf L_{\oo} = (\Phi L) _{\oo}$.
  
\end{proof}

\begin{rem}\label{perp}
	Consider the bijection $f_{\oo, \mathbf v}$ for $\oo =k$. Since the source of this bijection is non-empty, it follows that $\Fil ^1 (\Phi L)_k$ is orthogonal to the image of $\mathbf v$ in $(\Phi L) _{k}$. This observation also follows from the Remark \ref{rem: not used} as $\mathbf v \subset L \cap \Phi L$. 
\end{rem}

\subsection{Reducedness of minuscule special cycles}
\begin{prop}\label{prop:Wp2points}
	Let $\Lambda$ be a $\ZZ_p$-lattice in $V_K^{\Phi}$ with $p^i \Lambda \subset \Lambda^{\vee} \subset \Lambda$ for some $i \in \ZZ_{\geq 1}$. (Equivalently, $\Lambda^{\vee}$ has invariant $(r_1,\cdots, r_n) $ such that $ i\geq r_1\geq r_2\geq \cdots \geq r_n \geq 0$.) Then the special cycle $\mathcal Z(\Lambda^{\vee})$ defined by $\Lambda ^{\vee}$ has no $ (W/p^{i+1})$-points. In particular, taking $i=1$ we see that $\RZ_{\Lambda} (W/p^2) = \varnothing$ for any vertex lattice $\Lambda$, or equivalently $\mathcal Z (\mathbf v) (W/p^2) = \varnothing$ for any minuscule $\mathbf v$. 
\end{prop}
\begin{proof}
	Suppose there exists $x\in \mathcal Z(\Lambda^{\vee} ) (W/p^{i+1})$. Let $x_0 \in \mathcal Z(\Lambda^{\vee} ) (k)$ be induced by $x$ under the reduction map $W/p^{i+1} \to W/p =  k$. Under (\ref{eq:RZk}) $x_0$ determines a special lattice $L$. By Remark \ref{rem: v and L}, $\Lambda ^{\vee}_W \subset L \cap \Phi L$. Note that $W/p^{i+1} \to k$ is a surjection whose kernel admits nilpotent divided powers. By Theorem \ref{thm: MP}, the existence of the lift $x$ of $x_0$ implies that there exists an isotropic line $\mathcal L$  (over $W/p^{i+1}$) in $(\Phi L)_{ W/p^{i+1}}$ lifting $\Fil^1 (\Phi L)_{k}$ and such that $\mathcal L$ is orthogonal to the image of $\Lambda^{\vee}$ in $(\Phi L)_{W/p^{i+1}}$. Let $l \in \Phi L$ be a lift of a generator of $\mathcal L$. Then $\lprod{l,\lambda} \in p^{i+1} W$ for all $\lambda \in \Lambda^{\vee}$. It follows that $p^{-(i+1)} l \in \Lambda_W $. Hence $p^{-1 } l \in p^i \Lambda_W \subset (\Lambda^{\vee})_W \subset \Phi L$, i.e. $l \in p \Phi L$. This contradicts with the fact that $\mathcal L$ maps to a non-zero line in $(\Phi L)_{k}$.    
\end{proof}

\subsubsection{}\label{compute tangent space}
	Let $u \in V_K^{\Phi} -\set{0}$. Suppose $x_0 \in \mathcal Z(u)(k)$. Let $T= k[\epsilon]/\epsilon^2$ be the ring of dual numbers over $k$. We equip $T$ with the map $T\to k, \epsilon \mapsto 0$, which has its kernel $(\epsilon) $ admitting nilpotent divided powers (in a unique way). Thus Theorem \ref{thm: MP} can be applied to $\oo=T$. 
	
	Let $\mathcal T_{x_0} \RZ_k$ and $\mathcal T_{x_0} \mathcal Z(u)_k $ be the tangent spaces at $x_0$ to $\RZ_k = \RZ\times_{\Spf W } \spec k$ and to $\mathcal Z(u) _k = \mathcal Z(u) \times _{\Spf W} \spec k$ respectively. We will always take the point of view that $\mathcal T_{x_0} \RZ_k$ is the preimage of $\set{x_0}$ under the reduction map $\RZ (T) \to \RZ(k)$. Similarly for $\mathcal T_{x_0} \mathcal Z(u) _k$. We compute $\mathcal T_{x_0} \RZ_k$ and $\mathcal T_{x_0} \mathcal Z(u)_k $ explicitly in the following. The result is given in Corollary \ref{conclusion of tgt}. 
	
	Let $L$ be the special lattice associated to $x_0$ under (\ref{eq:RZk}). Since $x_0 \in \mathcal Z(u) (k)$, we have $u \in L \cap \Phi L$ by Remark \ref{rem: v and L}. Let $\bar u$ be the image of $u$ in $(\Phi L)_k$. Let $\Fil^1 (\Phi L)_k$ be as in Definition \ref{defn of Fil^1}. By Remark \ref{perp} we know that $\bar u$ is orthogonal to $\Fil^1 (\Phi L)_k$. 
	
	 Define $\mathscr D$ to be the set of isotropic lines in $(\Phi L)_T$ lifting $\Fil^1 (\Phi L)_k$. Define $\mathscr D_u $ to be the subset of $\mathscr D$ consisting of lines which are in addition orthogonal to the image of $u$ in $(\Phi L)_T$. Let 
	 \begin{align}\label{app of MP} \mathscr G= f_T: \mathcal T_{x_0}
	 \RZ_k \isom \mathscr  D.
	 \end{align} be the bijection given in Theorem \ref{thm: MP}. By the same theorem it restrict to a bijection 
	$$\mathcal T_{x_0} \mathcal Z(u)_k \isom \mathscr  D_u 
.$$	
	\begin{defn}
 We identify $(\Phi L)_T$ with $(\Phi L)_k \otimes _k T$. Fix a $k$-generator $v_0 $ of $\Fil^1 (\Phi L)_k$.  
 Define a map $$ \tilde {\mathscr F }: (\Phi L)_k  \to \set{\mbox{$T$-submodules of } (\Phi L)_T } $$ $$w\mapsto \mathrm{span}_T \set{ v_0 \otimes_k 1 + w\otimes_k \epsilon} . $$  
	\end{defn}
	\begin{lem}
		$\tilde{\mathscr F}$ factors through $(\Phi L)_k / \Fil^1 (\Phi L )_k$, and its image consists of $T$-module direct summands of $(\Phi L)_T$ of rank one.
	\end{lem}
	\begin{proof}
		For any $\lambda \in k$, we have 
		$$ v_0 \otimes 1 + (w+\lambda v_0)\otimes \epsilon = (1+ \lambda\epsilon ) (v_0 \otimes 1 + w\otimes \epsilon),$$ and $1+\lambda \epsilon \in T^{\times}$. Hence $\tilde{\mathscr F} $ factors through $(\Phi L)_k / \Fil^1 (\Phi L) _k$. For any $w\in (\Phi L)_k$, we know that $\tilde {\mathscr F} (w)$ is a free module of rank one by definition. It remains to show that $\tilde {\mathscr F} (w)$ is a direct summand of $(\Phi L)_T$. Let $A$ be a $k$-vector space complement of $\Fil^1(\Phi L) _k$ inside $(\Phi L)_k$. We easily check that the following $T$-submodule of $(\Phi L)_T$ is a $T$-module complement of $\tilde {\mathscr F}(w)$: 
		$$ \mathrm{span}_T \set{v' \otimes 1 + w\otimes \epsilon | v' \in A}.$$
	\end{proof}
	\begin{cor}\label{defn of mathscr F} The map $\tilde{\mathscr F}$ induces a bijection of sets: 
		$$   \mathscr F: ( \Fil ^1 (\Phi L)_k)^{\perp} / \Fil^1 (\Phi L)_k  \isom \mathscr D.$$ Moreover, $\mathscr F$ restricts to a bijection $$\set{\bar u, \Fil ^1 (\Phi L)_k} ^{\perp} /\Fil^1 (\Phi L)_ k \isom \mathscr D_u.$$
	\end{cor}\begin{proof}
	Since $\lprod{v_0,v_0} = 0 \in k$, the condition that $\tilde {\mathscr F}(w)$ is isotropic is equivalent to $\lprod{w, v_0} = 0 \in k$. Since $v_0$ is orthogonal to $\bar u$, the condition that $\tilde {\mathscr F}(w)$ is orthogonal to the image of $u$ in $(\Phi L)_T$ is equivalent to $\lprod{w , \bar u} = 0 \in k$. 
\end{proof}
 
\begin{lem}\label{k linear}Let $\mathscr G$ be as in (\ref{app of MP}) and let $\mathscr F$ be as in Corollary \ref{defn of mathscr F}. The map $$\mathscr G^{-1} \circ \mathscr F: (\Fil^1 (\Phi L)_k)^{\perp }/\Fil^1 (\Phi L)_k \to  \mathcal T_{x_0} \RZ_k $$ \ is $k$-linear. 
\end{lem}
\begin{proof}
	The proof is a routine check, using the functorial property stated in Theorem \ref{thm: MP}.
	
	We first recall the $k$-vector space structure on $\mathcal T_{x_0} \RZ_k$, from the point of view that $\mathcal T_{x_0} \RZ_k$ is the preimage of $\set{x_0}$ under the map $\RZ(T) \to \RZ(k)$. 
	
\textbf{Scalar multiplication}: Given a tangent vector $v \in \mathcal T_{x_0} \RZ_k$ corresponding to $v_T \in \RZ(T)$ and given a scalar $\lambda \in k$, the tangent vector $\lambda v$ corresponds to the following element $(\lambda v)_T$ of $\RZ(T)$: the image of $v_T$ under $\RZ(T) \xrightarrow{T\to T, \epsilon \mapsto \lambda \epsilon} \RZ(T)$. We see that $(\lambda v)_T$ is indeed a preimage of $x_0$.

\textbf{Addition}: Let $v_1, v_2\in \mathcal T_{x_0} \RZ_k$ be two tangent vectors. Let $T_i = k[\epsilon_i]/\epsilon _i ^2, i=1,2$ be two copies of $T$. We represent $v_i$ as an element $(v_i)_{T_i}$ in $\RZ(T_i)$ that reduces to $x_0 \in \RZ(k)$, for $i=1,2$. Let $\tilde T$ be the fiber product of $T_1$ and $T_2$ over $k$, in the category of $k$-algebras. Namely, $\tilde T = k[\epsilon_1,\epsilon_2]/(\epsilon_1,\epsilon_2) ^2$. Let $\delta$ be the $k$-algebra map 
$$\delta : \tilde T \to T, ~ \epsilon_1\mapsto \epsilon,~\epsilon_2 \mapsto \epsilon.$$  
  By the fact that $\tilde T$ is the fiber product of $T_1$ and $T_2$, there is a canonical bijection 
 \begin{align}\label{fiber prod}
 \RZ(T_1) \times \RZ(T_2) \isom \RZ(\tilde T).
 \end{align}
Denote by $v_1 \tilde + v_2$ the image of $((v_1)_{T_1}, (v_2) _{T_2} )$ in $\RZ(\tilde T )$ under the above bijection. Then the tangent vector $v_1+v_2$ corresponds to the following element $(v_1+v_2)_T$ of $\RZ(T)$: the image of $v_1 \tilde + v_2$ under $\delta_*: \RZ(\tilde T) \to \RZ(T)$. This last element is indeed a preimage of $x_0$. 
 
 We now check that $\mathscr G^{-1} \circ \mathscr F$ is $k$-linear. We first check the compatibility with scalar multiplication. For any $\lambda \in k$ and $w\in( \Fil ^1 (\Phi L)_k)^{\perp}$, we have $\tilde {\mathscr F} (w) = \mathrm{span}_T\set{v_0 \otimes 1 + w\otimes _k \epsilon}$ and $\tilde {\mathscr F} (\lambda w) = \mathrm{span}_T\set{v_0 \otimes 1 + \lambda w\otimes _k \epsilon}$. Let $m_{\lambda}$ denote the map $ T\to T, \epsilon \mapsto \lambda \epsilon$. Then we have $\tilde{\mathscr F } (w) \otimes _{T,m_{\lambda}} T = \tilde {\mathscr F} (\lambda w)$ as submodules of $(\Phi L)_T$. By the functoriality in $\oo$ stated in Theorem \ref{thm: MP}, we know that for all $d\in \mathscr D$, the element $\mathscr G ^{-1} (  d\otimes _{T,m_{\lambda}}T )\in \RZ(T) $ is equal to the image of $\mathscr G^{-1} (d)$ under $\RZ(T) \xrightarrow{(m_{\lambda})_*} \RZ(T)$. It follows that $(\mathscr G^{-1} \circ \mathscr F)(\lambda w)$ is equal to $\lambda$ times the tangent vector $(\mathscr G^{-1} \circ \mathscr F)( w)$.
 
 We are left to check the additivity of $\mathscr G^{-1} \circ \mathscr F$. Let $w_1, w_2 \in ( \Fil ^1 (\Phi L)_k)^{\perp}$. Let $\mathscr D_i, \mathscr F_i,\mathscr G_i $ be the analogues of $\mathscr D, \mathscr F, \mathscr G$ respectively with $T$ replaced by $T_i$, for $i=1,2$. Also let $f_{\tilde T}$ be as in Theorem \ref{thm: MP} (with $\oo = \tilde T$, where $\ker(\tilde T\to k) $ is equipped with the unique nilpotent divided power structure.) Let $d_i : = \mathscr F_i (w_i), ~ i =1,2$. Then $d_i = \mathrm{span}_{T_i} (v_0\otimes 1 + w_i \otimes \epsilon_i)$.
  We easily see that the assertion $(\mathscr G^{-1} \circ \mathscr F) (w_1 +w_2) = (\mathscr G^{-1} \circ \mathscr F)(w_1) + (\mathscr G^{-1} \circ \mathscr F)(w_2) $
follows from the following claim:

\textbf{Claim.} Under (\ref{fiber prod}), the element 
$ (\mathscr G_1^{-1} (d_1),\mathscr G_2^{-1} (d_2) )$
 is sent to the element 
 $$ f_{\tilde T}^{-1} (\mathrm {span}_{\tilde T} \set{v_0\otimes 1 + w_1\otimes \epsilon_1 + w_2\otimes \epsilon_2}).$$

 We now prove the claim. Let $\tilde d$ be such that the element $ (\mathscr G_1^{-1} (d_1),\mathscr G_2^{-1} (d_2) )$ is sent under (\ref{fiber prod}) to $f_{\tilde T}^{-1} (\tilde d)$.  Thus $\tilde d $ is an isotropic line in $ (\Phi L) _{\tilde T}$. By the functoriality stated in Theorem \ref{thm: MP} and the functorial definition of (\ref{fiber prod}), we see that $\tilde d$ is characterized by the condition that $\tilde d \otimes _{\tilde T} T_i = d_i, ~ i = 1,2,$ where the tensor product is with respect to the the structure map $\tilde T\to T_i$ expressing $\tilde T$ as the fiber product of $T_1,T_2$ (i.e. reduction modulo $\epsilon_j$ for $j \neq i$). Using this characterization of $\tilde d$, we see that $\tilde d$ is as predicted in the claim.
\end{proof}
\begin{cor}\label{conclusion of tgt}
	The tangent space $\mathcal T_{x_0} \RZ_k$ is isomorphic to $$(\Fil^1 (\Phi L)_k)^{\perp }/\Fil^1 (\Phi L)_k.$$ Under this isomorphism, the subspace $\mathcal T_{x_0} \mathcal Z(u)_k$ of $\mathcal T_{x_0} \RZ_k$ is identified with 
	$$\set{\bar u, \Fil ^1 (\Phi L)_k} ^{\perp} /\Fil^1 (\Phi L)_ k.$$
\end{cor}
\begin{proof}
	This follows from Corollary \ref{defn of mathscr F}, Lemma \ref{k linear}, and the bijectivity of $\mathscr G^{-1}$ asserted in Theorem {\ref{thm: MP}}. 
\end{proof}
\begin{lem}\label{linear alg}
	Let $\Lambda\subset V^{\Phi} _K$ be a vertex lattice. Let $L$ be a self-dual $W$-lattice in $V_K$ such that $\Lambda^{\vee}_W \subset L \subset \Lambda_W$. Let $A$ be the image of $\Lambda^{\vee}_W$ in $L_k$. Then the following statements hold.
	\begin{enumerate}
		\item $\dim_k \Lambda_W/ L = \dim _k L /\Lambda_W^{\vee}$. Here both spaces are vector spaces over $k$ because $p \Lambda_W \subset \Lambda_W^{\vee} \subset L$ and $p L \subset p \Lambda_W \subset \Lambda_W^{\vee}$.  
		\item $A \supset A^{\perp}$. Here $A^{\perp}$ is the orthogonal complement of $A$ in $L_k$.
	\end{enumerate}
\end{lem}
\begin{proof}
	\textbf{(1)} Consider the $W$-bilinear pairing $$\Lambda_W \times \Lambda_W \to W $$ $$(x,y) \mapsto p \lprod{x,y},$$ where $\lprod{,}$ is the $K$-bilinear form on $V_K^{\Phi } \otimes _{\QQ_p} K = V_K$. We get an induced $k$-quadratic space structure on $\Lambda_W/\Lambda_W^{\vee} $. The image of $L$ in $\Lambda_W/\Lambda_W^{\vee}$ is equal to the orthogonal complement of itself, i.e. it is a Lagrangian subspace. Claim (1) follows.
	
	\textbf{(2)} By definition $A^{\perp}$ is the image in $L_k$ of the $W$-submodule $p\Lambda_W^{\vee \vee} = p \Lambda_W$ of $L$.  We have $p\Lambda_W \subset \Lambda_W^{\vee}$, so $A^{\perp} $ lies in the image of $\Lambda_W^{\vee}$ in $L_k$, which is $A$.   
\end{proof}

\begin{prop}\label{prop: calc of dim of tgt} Let $\Lambda \subset V_K^{\Phi}$ be a vertex lattice of type $t$ (so $t\geq 2$ is even). For all $x_0 \in  \RZ_{\Lambda} (k)$, we have 
	$$\dim_k \mathcal T_{x_0} \RZ_{\Lambda,k}  = t/2-1. $$
\end{prop}
\begin{proof} 
	Let $L$ be the special lattice associated to $x_0$ under (\ref{eq:RZk}), and let $\Fil^1 (\Phi L)_k$ be as in Definition \ref{defn of Fil^1}. Then $\Lambda^{\vee} _W \subset L \cap \Phi L $. Denote by $A$ the image of $\Lambda_W^{\vee}$ in $(\Phi L)_k$. Then $A$ is orthogonal to $\Fil^1 (\Phi L)_k$ by Remark \ref{perp}. By Corollary \ref{conclusion of tgt}, we have an isomorphism of $k$-vector spaces $$\mathcal T_{x_0} \RZ_{\Lambda,k} \cong  \set{A, \Fil^1 (\Phi L)_k} ^{\perp} / \Fil^1 (\Phi L)_k .$$  Since $A$ is orthogonal to $\Fil^1 (\Phi L)_k$, we have $A \supset \Fil^1 (\Phi L)_k$ by Lemma \ref{linear alg} applied to the self-dual $W$-lattice $\Phi L$. Therefore $\mathcal T_{x_0}\RZ_{\Lambda,k} \cong A^{\perp} / \Fil^1 (\Phi L)_k$. Since the bilinear pairing on $(\Phi L)_k$ is non-degenerate, we have $\dim _k \mathcal T_{x_0}\RZ_{\Lambda,k} = \dim_k (\Phi L)_k - \dim_k A - 1 = \dim_k (\Phi L/\Lambda_W^{\vee}) -1$. By claim (1) in Lemma \ref{linear alg} (applied to $\Phi L$), we have $\dim_k (\Phi L/\Lambda_W^{\vee}) = t/2$.
\end{proof}

\begin{cor}\label{cor:regularity}
	Let $\Lambda \subset V^{\Phi} _K$ be a vertex lattice. The formal scheme $\RZ_{\Lambda} \times_{\Spf W} \spec k$ is regular.
\end{cor}
\begin{proof}Let $t$ be the type of $\Lambda$.	Denote $X: = \RZ_{\Lambda} ^{\mathrm{red} } $ and $Y: = \RZ_{\Lambda} \times_{\Spf W} \spec k$. Then $X$ is a formal subscheme of $Y$ over $k$. Recall from \S \ref{sec:str} that $X$ is a smooth $k$-scheme of dimension $t/2-1$. It follows that for all $x_0 \in Y(k)$, the complete local ring of $Y$ at $x_0$ is of dimension $\geq t/2-1$. By Proposition \ref{prop: calc of dim of tgt}, the tangent space of $Y$ at $x_0$ has $k$-dimension equal to $t/2 -1$. Hence $Y$ is regular at $x_0$.        
\end{proof}

\begin{thm}\label{thm:reducedminusculecyle}
	Let $\Lambda \subset V^{\Phi} _K$ be a vertex lattice. Then $\RZ_{\Lambda} = \RZ_{\Lambda}^{\mathrm{red}}$ and is of characteristic $p$.
\end{thm}
\begin{proof}
 $\RZ_{\Lambda}$ does not admit $W/p^2$-points (Proposition \ref{prop:Wp2points}) and its special fiber is regular (Corollary \ref{cor:regularity}). It follows from \cite[Lemma 10.3]{RTZ} that $\RZ_{\Lambda}$ is equal to its special fiber. Being regular itself, $\RZ_{\Lambda} $ is reduced.
\end{proof}

\section{The intersection length formula}\label{sec:inters-length-form}

\subsection{The arithmetic intersection as a fixed point scheme}\label{subsection: lth}
Recall from \S \ref{sec:goal} that we are interested in computing the intersection of $\RZ^g$ and $\delta(\RZ^{\flat})$, for $g\in J_b (\QQ_p)$.

\begin{prop}\label{prop:concentration weak}
	Assume $g\in J_b(\mathbb{Q}_p)$ is regular semisimple. Then $\delta(\RZ^\flat)\cap \RZ^g$ is contained in $\mathcal{Z}(\mathbf{v}(g))$, where $\mathbf v (g) = ( x_n, gx_n,\cdots, g^{n-1} x_n)$.
\end{prop}

\begin{proof}
	By Lemma \ref{lem:specialdivisor}, we have $\delta(\RZ^\flat)\subseteq \mathcal{Z}(x_n)$. Hence $\delta(\RZ^\flat)\cap \RZ^g\subseteq \mathcal{Z}(x_n)\cap \RZ^g\subseteq \mathcal{Z}(gx_n)$ by the definition of special cycles. Repeating this procedure we obtain $$\delta(\RZ^\flat)\cap\RZ^g\subseteq \mathcal{Z}(x_n)\cap \mathcal{Z}(gx_n)\cap\cdots\cap\mathcal{Z}(g^{n-1}x_n)=\mathcal{Z}(\mathbf{v}(g)).\qedhere$$ 
\end{proof}
\begin{cor}\label{concentration}
Assume $g\in J_b(\QQ_p)$ is regular semi-simple and minuscule. Then $$\delta(\RZ^\flat)\cap \RZ^g \subset \RZ _{L(g) ^{\vee}} = \RZ_{L(g) ^{\vee}} ^{\mathrm{red}}.$$ In particular, $\delta(\RZ^\flat)\cap \RZ^g$ is a scheme of characteristic $p$.
\end{cor}
\begin{proof}
	The first statement is an immediate consequence of Remark \ref{rem:minsculespecialcyle}, Theorem \ref{thm:reducedminusculecyle}, and Proposition \ref{prop:concentration weak}. Now both $\delta (\RZ^{\flat})$ and $\RZ^g$ are closed formal subschemes of $\RZ$, so $\delta (\RZ^{\flat})\cap \RZ^g$ is a closed formal subscheme of the scheme $\RZ _{L(g) ^{\vee}} = \RZ_{L(g) ^{\vee}} ^{\mathrm{red}}$ of characteristic $p$. Hence $\delta (\RZ^{\flat})\cap \RZ^g$ is its self a scheme of characteristic $p$. 
\end{proof}
\subsubsection{}\label{notation}
In the rest of this section we will fix $g \in J_b(\QQ_p)$ regular semisimple and minuscule, and assume $\RZ^g \neq \varnothing$. Take $\Lambda:=L(g)^{\vee}$. Then $\Lambda$ is a vertex lattice stable under $g$, cf. Remark \ref{rem:regular}. We are interested in computing the intersection length of $\delta(\RZ^{\flat})$ and $\RZ^g $ around a $k$-point of intersection. Recall the isomorphism (\ref{eq:rzlambda}) between $p^{\ZZ}\backslash \RZ_{\Lambda} ^{\mathrm{red}}$ (which we now know is just $p^{\ZZ}\backslash \RZ_{\Lambda}$) and $S_{\Lambda}$. Recall from \S \ref{sec:variety-s_lambda} that $S_{\Lambda}$ is a projective smooth variety over $k$ of dimension $t_{\Lambda} /2 -1$. We write $d = t_{\Lambda}/2$. Let $\Omega_0 = \Lambda / \Lambda^{\vee} $ and $\Omega = \Omega_0 \otimes _{\FF_p}  k = \Lambda_{W} / \Lambda_W^{\vee}$. Let $\lprod{,}$ be the $k$-bilinear form on $\Omega$ (cf. \S \ref{sec:vertex-lattices}). Let $\mathbb G = \SO(\Omega) , \mathbb G_0 = \SO(\Omega_0)$. Let $\bar g$ be the induced action of $g$ on $\Omega$. Then $\bar g  \in \mathbb G_0(\FF_p) \subset \mathbb G(k)$. \ignore{Recall that $$S_{\Lambda} (k) = \set{\mbox{Lagrangians } \mathcal L \subset \Omega,~ \dim (\mathcal L + \Phi \mathcal L) =d+1 }$$$$ = \set{(\mathcal L_{d-1} ,\mathcal L_d) | \mathcal L_d \subset \Omega \mbox{~is Lagrangian,~} \dim \mathcal L_{d-1} = d-1, \mathcal L_{d-1} \subset \mathcal L_d \cap \Phi \mathcal L_d }.$$}

There is a natural action of $\bar g$ on $S_{\Lambda}$ via its action on $\Omega$. On $R$-points $\bar g$ sends $(\mathcal L_{d-1}, \LL_d)$ to $(\bar g \LL_{d-1}, \bar g \LL_d)$. The latter is indeed a point of $S_{\Lambda}$ because $\bar g \Phi = \Phi \bar g $ by the fact that $\bar g \in \mathbb G_0 (\FF_p)$. The following proposition allows us to reduce the study of intersection multiplicities to the study of the non-reduced structure of $S_{\Lambda} ^{\bar g}$. 

\begin{prop}\label{prop:nonreduced}
 $p^\mathbb{Z}\backslash(\delta (\RZ^{\flat})\cap \RZ^g)\cong S_\Lambda^{\bar g}$.
\end{prop}

\begin{proof}
In view of Theorem \ref{thm:reducedminusculecyle}, Corollary \ref{concentration} and the observation that the isomorphism (\ref{eq:rzlambda}) induces an isomorphism $p^{\ZZ} \backslash (\RZ_{\Lambda}^{\mathrm{red} } )^g \isom S_{\Lambda} ^{\bar g} $, it suffices to show 
$$ (p^{\ZZ} \backslash \RZ_{\Lambda}^{\mathrm{red}} ) \cap  ( p^{\ZZ} \backslash \delta(\RZ^{\flat})) = (p^{\ZZ} \backslash \RZ_{\Lambda}^{\mathrm{red}} ). $$ Since both $p^{\ZZ} \backslash \RZ_{\Lambda}^{\mathrm{red}} $ and $ p^{\ZZ} \backslash \delta(\RZ^{\flat})$ are closed formal subschemes of $p^{\ZZ} \backslash \RZ$ and since $p^{\ZZ} \backslash \RZ_{\Lambda}^{\mathrm{red}} $ is a reduced scheme, it suffices to check that $$ p^{\ZZ} \backslash \RZ_{\Lambda} ^{\mathrm{red}} (k) \subset p^{\ZZ } \backslash \delta (\RZ^{\flat}) (k). $$ Now the left hand side consists of special lattices $L$ containing $\Lambda^{\vee}$, and the right hand side consists of special lattices $L$ containing $x_n$ (cf. (\ref{special lattices in RZLambda}) and Lemma \ref{lem:Lg}). We finish the proof by noting that by definition $x_n \in \Lambda^{\vee} = L(g) $. 
\end{proof}


Proposition \ref{prop:nonreduced} reduces the intersection problem to the study of $S_{\Lambda}^{\bar g}$. 

\subsection{Study of $S_{\Lambda}^{\bar g}$}

We continue to use the notation in \S \ref{subsection: lth}. 
We adopt the following notation from \cite[\S 3.2]{HPGU22}.
\begin{defn}
	Let $\OGr(d-1)$ (resp. $\OGr(d)$) be the moduli space of totally isotropic subspaces of $\Omega$ of dimension $d-1$ (resp. $d$). For a finite dimensional vector space $W$ over $k$ and an integer $l$ with $0\leq l \leq \dim W$, we write $\Gr(W,l)$ for the Grassmannian classifying $l$-dimensional subspaces of $W$. Thus for $j \in \set{d-1, d}$ and any $k$-algebra $R$, we have 
	$$\OGr (j) (R) = \set{\mbox{$R$-module local direct summands of $\Omega\otimes _k R$ of local rank $j$ which are totally isotropic}}.$$
	Also $$\Gr(W,l) (R) = \set{\mbox{$R$-module local direct summands of $W\otimes _k R$ of local rank $k$}}.$$
	\ignore{Let $\OGr(d-1, d) \subset \OGr(d-1) \times \OGr(d)$ be the subscheme defined by the incidence relation. Namely, it is the locus where the $(d-1)$-dimensional isotropic subspace is contained in the $d$-dimensional isotropic subspace. }
\end{defn}

 \begin{defn}
 	If $A$ is a finite dimensional $k$-vector space, we write $\aff{A}$ for the affine space over $k$ defined by $A$. Thus for a $k$-algebra $R$ we have $\aff{A} (R) = A\otimes_k R$. 
 \end{defn}
 \ignore{\begin{defn}
 	Let $V_1\subset V_2 \subset V_3$ be $k$-vector spaces. Let $A$ be a subspace of $\Hom(V_1, V_3)$ and $B$ be a subspace of $\Hom (V_2, V_3)$. Write $(A\times B)^{\mathrm{comp}}$ for the subspace of $A\times B$ consisting of compatible elements, i.e. elements $(\phi,\psi) \in A\times B$ such that $\psi|_{V_1} = \phi$. Write $\aff{A} \times^{\mathrm{comp}} \aff{B}$ for $\aff{(A\times B)^{\mathrm{comp}} }$. 
 \end{defn}}
 \begin{defn}Let $\mathcal L_d,\mathcal M_d $ be Lagrangian subspaces of $\Omega$ such that $\Omega = \mathcal L_d \oplus \mathcal M_d$. we write $\Hom_{\anti} (\mathcal L_d, \mathcal M_d)$ for the space of anti-symmetric $k$-linear maps $\mathcal L_d \to \mathcal M_d$. Here we say $\phi: \mathcal L_d \to \mathcal M_d$ is anti-symmetric if the bilinear form $\LL_d \times \LL_d \to k, (x,y) \mapsto \lprod{x, \phi y}$ is anti-symmetric.
 	 \end{defn}
 	 
 	 \subsubsection{}\label{recall of Grass}
Recall that in general, if $A$ is a finite dimensional vector space over $k$ and $B$ is a subspace, then we can construct a Zariski open of the Grassmannian $\Gr (A,\dim B)$ as follows. Choose a subspace $C$ of $A$ such that $A = B \oplus C$. Then there is an open embedding $\iota_{B,C}: \aff{\Hom_k (B, C)} \to \Gr (A, \dim B)$ which we now describe. For any $k$-algebra $R$ and any $R$-point $\phi$ of $\aff{\Hom _k(B,C)}$, we view $\phi$ as an element of $\Hom _k(B, C) \otimes R = \Hom _R(B\otimes R, C\otimes R)$. Then $\iota_{B,C}$ maps $\phi$ to the $R$-point of $\Gr(A, \dim B)$ corresponding the following $R$-submodule of $A$: 
\begin{align}\label{subspace corrsp to phi}
\set{x+ \phi(x)| x \in B\otimes R}.
\end{align}
For details see for instance \cite[Lecture 6]{Joe}. In the following we will think of $\aff{\Hom_k(B,C)}$ as a Zariski open of $\Gr(A,\dim B)$, omitting $\iota_{B,C}$ from the notation.

 \begin{lem}\label{local structure of OGr}
 Let $\mathcal L_d,\mathcal M_d $ be complementary Lagrangian subspaces of $\Omega$ over $k$. Then $$ \OGr(d) \times _{\Gr(\Omega, d)} \aff{\Hom (\LL_d, \mathcal M_d)} = \aff{\Hom_{\anti} ( \mathcal L_d ,\mathcal M_d) }.$$ In particular, the $k$-point $\LL_d$ in $\OGr(d)$ has an open neighborhood of the form $\aff{\Hom_{\anti} ( \mathcal L_d ,\mathcal M_d) }$.
 \end{lem}
 \begin{proof}
 	Let $R$ be a $k$-algebra and $\phi$ an $R$-point of $\aff{\Hom (\LL_d, \mathcal M_d)}$. Then the submodule (\ref{subspace corrsp to phi}) (for $B = \LL_d$) is Lagrangian if and only if for all $x\in B\otimes R$, $$\lprod{x+ \phi (x) , x + \phi(x)} =0. $$ But we have $\lprod{x , x} = \lprod{ \phi (x), \phi(x)} = 0 $ since $\LL_d\otimes R$ and $\mathcal M_d \otimes R$ are both Lagrangian. Hence (\ref{subspace corrsp to phi}) is Lagrangian if and only if $\lprod{x, \phi(x)} =0 $ for all $x\in \mathcal M_d \otimes R$.
 	  \end{proof}
\subsubsection{}\label{take x_0}
It follows from the assumptions we made on $\bar g \in \mathbb G(k)$ in \ref{notation} that its characteristic polynomial on $\Omega$ is equal to its minimal polynomial on $\Omega$ (cf. Remark \ref{rem:regular}). In general this property is equivalent to the property that in the Jordan normal form all the Jordan blocks have distinct eigenvalues. From now on we let $x_0  = (\mathcal L_{d-1} ,\mathcal L_d) \in S_{\Lambda} (k) $ be an element fixed by $\bar g$. Then $\Phi \mathcal L_d \subset \Omega$ is also stable under $\bar g$. If we identify $\Omega/ \Phi \mathcal L_d$ with $(\Phi \mathcal L_d)^*$ (the $k$-vector space dual) using the bilinear form on $\Omega$, the action of $\bar g$ on $\Omega/ \Phi \mathcal L_d$ is equal to the inverse transpose of $\bar g|_{\Phi \mathcal L_d}$. It follows that the minimal polynomial (resp. characteristic polynomial) of $\bar g$ on $\Omega$ is equal to the minimal polynomial (resp. characteristic polynomial) of $\bar g|_{\Phi \mathcal L_d}$ times its reciprocal. Hence $\bar g|_{\Phi \mathcal L_d}$ has equal minimal and characteristic polynomial, too.

\begin{defn}\label{defn of lambda c}
	Let $\lambda$ be the (nonzero) eigenvalue of $\bar g$ on the one-dimensional $\Phi \mathcal L_d/\mathcal L_{d-1}$. Let $c$ be the size of the unique Jordan block of eigenvalue $\lambda$ of $\bar g|_{\Phi \mathcal L_d}$.  
\end{defn}
 \subsubsection{}\label{defn of U and I}
Let $x_0 = (\mathcal L_{d-1} ,\mathcal L_d ) \in S_{\Lambda} (k)^{\bar g}$ as in \ref{take x_0}. Define 
$$Y: = \Gr(\Phi \LL_d, d-1 ) \times _k \OGr (d).$$
 Let $\mathcal I \subset Y$ be the sub-functor defined by the incidence relation, i.e. for a $k$-algebra $R$
  $$ \mathcal I (R) = \set{ (\LL_{d-1} ' , \LL_d') \in \Gr (\Phi \LL_d, d-1)  (R) \times \OGr (d) (R) ~|~  \LL_{d-1}' \subset \LL'_{d} }. $$
   The pair $(\LL_{d-1} ,\LL_{d} )$ defines a $k$-point in $\mathcal I$, which we again denote by $x_0$. It is well known that the incidence sub-functor of $\Gr(\Phi \LL_d, d-1) \times \Gr(\Omega, d)$ is represented by a closed subscheme, and it follows that $\mathcal I$ is a closed subscheme of $Y$.

Since $x_ 0 = (\LL_{d-1}, \LL_d) \in S_{\Lambda} (k)$ is fixed by $\bar g$, we have a natural action of $\bar g$ on $Y$, stabilizing $\mathcal I$ and fixing $x_0 \in \mathcal I$. Let $$\tilde{\mathcal R}: = \oo_{\mathcal I ,x_0},\quad \mathcal R: = \oo_{\mathcal I^{\bar g}, x_0},\quad \tilde {\mathcal S}: = \oo_{S_{\Lambda} , x_0},\quad\mathcal S: = \oo_{S_{\Lambda}^{\bar g} , x_0}$$ be the local rings at $x_0 $ of $\mathcal I, \mathcal I^{\bar g}, S_{\Lambda}, S_{\Lambda}^{\bar g}$ respectively. Let $$\tilde {\mathcal R}_p : = \tilde {\mathcal R} / \mathfrak m_{\tilde {\mathcal R}}^p,\quad \mathcal R_p : = \mathcal R/\mathfrak m_{\mathcal R}^p,\quad\tilde {\mathcal S}_p : = \tilde {\mathcal S}/\mathfrak m_{\tilde {\mathcal S}}^p,\quad \mathcal S_p: = \mathcal S/\mathfrak m_{\mathcal S}^p$$
be the above four local rings modulo the $p$-th powers of their respective maximal ideals.

The following lemma expresses the observation that $\mathcal I^{\bar g}$ may serve as a model for $S_{\Lambda}^{\bar g}$ locally around $x_0$. 

\begin{lem}\begin{enumerate}		
		\item There is a $k$-algebra isomorphism $ \tilde{\mathcal R} _p \cong \tilde {\mathcal S} _p$, equivariant for the $\bar g$-action on both sides. 
	\item There is a $k$-algebra isomorphism $ \mathcal R_p \cong \mathcal S_p$. 
	\end{enumerate}\label{loc mod} \end{lem}
\begin{proof}
	We first show (1). Let $(\LL_{d-1}' , \LL_{d}')$ be the tautological pair over $\tilde{\mathcal S} _p$ for the moduli problem $S_{\Lambda}$, and let $(\LL_{d-1}'' , \LL_{d}'')$ be the tautological pair over $\tilde{\mathcal R} _p$ for the moduli problem $\mathcal I$. Note that $$\Phi \LL_d' = (\Phi \LL_d) \otimes \tilde {\mathcal S} _p$$ as submodules of $\Omega \otimes_k \tilde {\mathcal S} _p $ because $\Phi: \tilde {\mathcal S} _p \to \tilde {\mathcal S} _p$ factors through the reduction map $\tilde {\mathcal S} _p \to k$. It follows that $(\LL_{d-1}',\LL_d')$ defines a point in $\mathcal I (\tilde{\mathcal S}_p)$ lifting $x_0 \in \mathcal I(k)$. Similarly, $$\Phi \LL_d'' = (\Phi \LL_d) \otimes \tilde {\mathcal R} _p$$ as submodules of $\Omega \otimes_k \tilde {\mathcal R} _p $, and hence $(\LL_{d-1}'',\LL_d'')$ defines a point in $S_{\Lambda} (\tilde {\mathcal R}_p)$ lifting $x_0 \in \mathcal S_{\Lambda}(k)$. The point in $\mathcal I(\tilde{\mathcal S}_p)$ and the point in $S_{\Lambda} (\tilde {\mathcal R} _p) $ constructed above give rise to inverse $k$-algebra isomorphisms between $\tilde{\mathcal R} _p$ and $\tilde {\mathcal S} _p$, which are obviously $\bar g$-equivariant. 
	
	(2) follows from (1), since $\mathcal R_p$ (resp. $\mathcal S_p$) is the quotient ring of $\tilde{\mathcal R}_p$ (resp. $\tilde{\mathcal S} _p$) modulo the ideal generated by elements of the form $r - \bar g \cdot r$ with $r\in \tilde{\mathcal R}_p$ (resp. $r\in \tilde{\mathcal S}_p$).   
 \end{proof}
\ignore{\begin{rem}
	The above proof also shows that $\oo_{S_{\Lambda} ^+, x_0}/\mathfrak m_{x_0}^l$ is isomorphic to $\oo_{\mathcal I , x_0} /\mathfrak m _{x_0} ^l$ for $1\leq l \leq p$. In particular, $S_{\Lambda} ^+$ and $\mathcal I$ have isomorphic tangent spaces at $x_0$. By Lemma \ref{eqn for I}, we see that the tangent space of $\mathcal I$ at $x_0$ is isomorphic to $(\Hom (\LL_{d-1}, \mathrm{span}_k (w_d)) \times \Hom_{\anti} (\LL_d, ~\mathrm{span}_k (w_1,\cdots, w_d)))^{\mathrm{comp}}$. The last vector space is isomorphic to $\Hom(\LL_{d-1}, \mathrm{span}_k (w_d)) $ via projection to the first factor. From this we recover Corollary \ref{grass version of tan}. 
\end{rem}}

\subsection{Study of $\mathcal{I}^{\bar g}$}\label{choice of basis}
Next we study $\mathcal{I}^{\bar g}$ by choosing certain explicit coordinates on $\mathcal{I}$. Choose a $k$-basis $v_1,\cdots, v_d, w_1,\cdots , w_d$ of $\Omega$, such that 
 \begin{itemize}
 	\item $\LL_{d-1}$ is spanned by $v_1,\cdots, v_{d-1}$.
 	\item $\LL_d$ is spanned by $v_1,\cdots, v_d$.
 	\item $\Phi \LL_d$ is spanned by $v_1,\cdots, v_{d-1} , w_d$.
 	\item $\lprod{v_i, v_j} = \lprod{w_i , w_j} = 0, \lprod{v_i,w_j} = \delta_{ij}.$
 \end{itemize}
 We will denote $$\hat v_i : = \begin{cases}
 v_i,~ 1\leq i\leq d-1 \\ 
 w_d, ~ i=d
 \end{cases}$$
 Also denote 
 $$\mathcal M_d : = \mathrm{span}_k (w_1,\cdots, w_d). $$

 For $1\leq i \leq d-1$, define an element $\phi_i \in \Hom (\LL_{d-1} , ~\mathrm{span}_k (w_d) )$ by 
 \begin{align}\label{phi_i}
 \phi_i (v_j) = \delta_{ij}  w_d.
 \end{align}
 Then $\phi_1,\cdots, \phi_{d-1}$ is a basis of $\Hom (\LL_{d-1} , ~\mathrm{span}_k (w_d) )$.

By \S \ref{recall of Grass} and Lemma \ref{local structure of OGr}, there is a Zariski open neighborhood of $x_0$ in $Y$, of the form $$\mathcal U: = \aff{\Hom (\LL_{d-1} ,  ~\mathrm{span}_k (w_d) )} \times \aff{\Hom_{\anti} (\LL_d,\mathcal M_d)} .$$
 
\begin{lem}
	\begin{enumerate}
		\item  Let $R$ be a $k$-algebra. Let $y \in \mathcal U(R)$, corresponding to $$(\phi,\psi) \in \Hom (\LL_{d-1} ,       \mathrm{span}_k(w_d))\otimes R \oplus \Hom_{\anti} (\LL_d, \mathcal M_d) \otimes R.$$ We view $\phi \in \Hom_R(\LL_{d-1}\otimes R, \mathrm{span}_R (w_d))$ and $\psi \in \Hom _R (\LL_d \otimes R, \mathcal M_d \otimes R)$. Then $y$ is in $\mathcal I$ if and only if $\psi|_{\LL_{d-1} \otimes R} = \phi$. 
		\item The projection to the first factor $\mathcal U \to \aff{\Hom (\LL_{d-1} ,\mathrm{span}_k (w_d))}$ restricts to an isomorphism 
		$$ \mathcal U \cap \mathcal I \isom \aff{\Hom (\LL_{d-1} ,\mathrm{span}_k (w_d))}. $$
	\end{enumerate}
\label{eqn for I}
\end{lem} 
\begin{proof}
	\textbf{(1)} We know that $y$ is in $\mathcal I$ if and only if for all $v\in \LL_{d-1}\otimes R$, there exists $v'\in \LL_d\otimes R$, such that $$v+ \phi(v)  = v'+ \psi(v') $$ as elements of $\Omega\otimes R$. Decompose $v' = v'_1 + v'_2$ with $v'_1 \in \LL_{d-1} \otimes R$ and $v'_2 \in \mathrm{span}_R (v_d)$. Then the above equation reads 
	$$ v-v'_1  = v'_2 + (\psi(v') - \phi(v)).$$ Since $v-v'_1 \in \LL_{d-1} \otimes R,~ v_2' \in \mathrm{span}_R (v_d), ~\psi(v')-\phi(v) \in \mathcal M_d \otimes R $, the above equation holds if and only if $v= v_1' , ~ v_2' = 0, ~\phi(v) = \psi(v)$. 
	Hence $y\in \mathcal I$ if and only if for all $v \in \LL_{d-1}\otimes R$ we have $\psi(v) = \phi(v)$. This proves (1). 
	
	\textbf{(2)} By (1), we know that $\mathcal U \cap \mathcal I$ is the affine subspace of $\mathcal U$ associated to the linear subspace of $$\Hom (\LL_{d-1} , \mathrm{span}_k (w_d)) \times \Hom_{\anti} (\LL_d, \mathcal M_d)$$ consisting of pairs $(\phi,\psi)$ such that $\psi|_{\LL_{d-1}} = \phi$. Call this subspace $A$. We only need to show that projection to the first factor induces an isomorphism $A \isom \Hom (\LL_{d-1} ,\mathrm{span}_k  (w_d))$.
	
	 Note that if $\psi \in \Hom_{\anti} (\LL_d, \mathcal M_d)$, then $\psi$ is determined by $\psi|_{\LL_{d-1}}$. This is because for each $1\leq i \leq d$, we have \begin{align}\label{psi v_d}
\lprod{\psi v_d, v_i} = \begin{cases}
- \lprod{v_d, \psi v_i} , ~  i \leq d-1\\
0, ~ i =d
\end{cases}
	\end{align} which means that $\psi(v_d)$ is determined by $\psi|_{\LL_{d-1}}$. Conversely, given $\phi \in \Hom (\LL_{d-1} , \mathrm{span}_k (w_d))$ we can construct $\psi \in \Hom_{\anti} (\LL_d, \mathcal M_d) $ such that $\psi|_{\LL_{d-1}} = \phi$ as follows. For $1\leq j \leq d-1$, define $\psi  (v_j)$ to be $\phi(v_j)$. Define $\psi(v_d)$ to be the unique element of $\mathcal M_d$ satisfying (\ref{psi v_d}). In this way we have defined a linear map $\psi: \LL_{d} \to \mathcal M_d$ such that $\psi|_{\LL_{d-1}} = \phi$. We now check that $\psi $ is anti-symmetric. We need to check that for all $1\leq i \leq j \leq d$, we have $\lprod{\psi v_j , v_i} = - \lprod{\psi v_i, v _j}$. If $j = d$, this is true by (\ref{psi v_d}). Suppose $j<d$. Then $\lprod{\psi v_j , v_i} = \lprod{\psi v_i, v_j} = 0$ because $\psi v_j , \psi v_i \in \mathrm{span}_k (w_d)$ and $\lprod{w_d, \LL_{d-1}} = 0$. Thus $\psi$ is indeed antisymmetric. It follows that $A \isom \Hom (\LL_{d-1}, \mathrm{span}_k (w_d))$.
\end{proof}
  \ignore{
 \subsubsection{}\label{discussion of fixed points in Y}
 Now assume that $x_0 \in S_{\Lambda}^{\bar g}(k)  .$ Then $\bar g$ stabilizes $\LL_d, \Phi \LL_d, \LL_{d-1}$. The natural action of $\bar g$ on $Y$ does not stabilize $\mathcal U$ in general, but we have a natural identification 
 \begin{align}\label{gU}
 \bar g \cdot  \mathcal U \cong \aff{\Hom (\LL_{d-1}, \mathrm{span}_k (\bar g w_d))} \times \aff{\Hom_{\anti} (\LL_d, \bar g \mathcal M_d)} \subset Y. 
 \end{align}
 Let $R$ be a $k$-algebra. If $y \in \mathcal U(R)$ is given by $(\phi,\psi)$, where $\phi \in \Hom (\LL_{d-1}, \mathrm{span}_k ( w_d))\otimes R$ and $\psi \in \Hom_{\anti} (\LL_d, \mathcal M_d ) \otimes R$, then $\bar g (y ) \in \bar g \cdot \mathcal U$ corresponds under (\ref{gU}) to $$(\bar g\circ \phi \circ (\bar g|_{\LL_{d-1}})^{-1}, \bar g \circ \psi \circ (\bar g|_{\LL_d})^{-1} ). $$ Denote the last pair by $(\leftidx{^{\bar g}} \phi, \leftidx {^{\bar g}} \psi).$ Then $y = \bar g y$ if and only if the following two conditions hold. \begin{align}
 \label{cond 1} \forall u \in \LL_{d-1}, \exists u' \in \LL_{d-1} \otimes R,  ~ u + \phi (u)  = u' + \leftidx{^{\bar g}}\phi (u') \\
  \label{cond 2} \forall v \in \LL_{d}, \exists v' \in \LL_{d} \otimes R, ~ v+ \psi (v) = v'+ \leftidx{^{\bar g}} \psi(v')
\end{align}}
From now on we assume $x_0 = (\LL_{d-1}, \LL_d)\in S_{\Lambda} ^{\bar g} (k)$.
\begin{defn}
	Write the matrix over $k$ of $\bar g$ acting on $\Phi \LL_d$ under the basis $\hat v_1,\cdots, \hat v_{d}$ (cf. \S \ref{choice of basis}) as 
	$$ \begin{pmatrix}
	H_1 & H_2 \\ H_3 & H_4
	\end{pmatrix},$$ where $H_1$ is of size $(d-1) \times (d-1)$, $H_2$ is of size $(d-1) \times 1$, $H_3$ is of size $1\times (d-1)$, and $H_4\in k$. 
\end{defn}
\begin{rem}\label{H_3=0}
Since $\bar g$ stabilizes $\LL_{d-1}$, we have $H_3=0$
\end{rem}

\begin{prop}\label{prop: fixed pts in Y}
	Let $R$ be a $k$-algebra and let $y =(\phi,\psi) \in \mathcal U(R)$. Represent $\phi$ as an $R$-linear combination $\phi = \sum_{i=1} ^{d-1} r_i \phi_i$ of the $\phi_i$'s (cf. (\ref{phi_i})), where $r_i \in R$. Write $\vec{r}$ for the row vector $(r_1,\cdots, r_{d-1}).$  
	\begin{enumerate}
		\item View $\phi$ as an element of $ \Gr (\Phi \LL_d, d-1)(R).$ It is fixed by $\bar g|_{\Phi \LL_d}$ if and only if 
		\begin{align}\label{key equation}
		\vec r ( H_1 + H_2 \vec r) = H_4 \vec r.
		\end{align}
		\ignore{\begin{align}\label{cond 1 easy}
		r_i = h_{dd} r_i \sum_{l=1}^{d-1} h^{ld} r_l +   h_{dd} \sum_{ l =1}^{d-1} h^{li} r_l ,~  1\leq i \leq d-1.
		\end{align}} 	
		\item Assume that $y \in \mathcal I(R)$ and that $\phi \in \Gr (\Phi \LL_d, d-1)$ is fixed by $\bar g|_{\Phi \LL_d}$. Then $\psi$, viewed as an element of $ \OGr(d)(R)$, is fixed by $\bar g$. In other words, $y$ is fixed by $\bar g$ in this case.
	\end{enumerate}
\end{prop}
\begin{proof} 
	\textbf{(1)}
	First we identify $(\Phi \LL_d)\otimes R$ with $R^{d-1}$ using the basis $\hat v_1,\cdots, \hat v_d$. 
	As a point of $\Gr(\Phi \LL_d, d-1)$, $\phi$ corresponds to the following submodule of $(\Phi \LL_d)\otimes R$: the image, i.e. column space, of the $R$-matrix 
	$$ \begin{pmatrix}
	I_{d-1} & 0 \\  \vec r & 0 
	\end{pmatrix}.$$
Hence $\phi \in \Gr (\Phi \LL_d, d-1)$ is fixed by $\bar g|_{\Phi \LL_d}$ if and only if the following two $R$-matrices have the same column space:
$$ A_1 : = \begin{pmatrix}
I_{d-1} & 0 \\  \vec r & 0 
\end{pmatrix} \mbox{ and } ~A_2:= \begin{pmatrix}
H_1 & H_2 \\ H_3 & H_4
\end{pmatrix}\begin{pmatrix}
I_{d-1} & 0 \\  \vec r & 0 
\end{pmatrix}. $$
Note that since $\begin{pmatrix}
H_1 & H_2 \\ H_3 & H_4
\end{pmatrix}$ is invertible, $A_1$ and $A_2$ have the same column space if and only if the column space of $A_2$ is contained in that of $A_1$.
Since $H_3 =0$ (cf. Remark \ref{H_3=0}), we have
$$A_2=  \begin{pmatrix}
H_1 + H_2 \vec r & 0 \\ H_4 \vec r & 0
\end{pmatrix}. $$
But we easily see that the column space of $\begin{pmatrix}
H_1 + H_2 \vec r & 0 \\ H_4 \vec r & 0
\end{pmatrix}$ is contained in that of $ \begin{pmatrix}
I_{d-1} & 0 \\  \vec r & 0 
\end{pmatrix}$ if and only if (\ref{key equation}) holds.

\textbf{(2)}
 Let $\OGr(d-1,d)$ be the incidence subscheme of $\OGr(d-1) \times \OGr(d)$. Consider the natural morphism $f: \mathcal I \to \OGr(d-1,d)$, $(\LL_{d-1}', \LL_d') \mapsto (\LL_{d-1}', \LL_d').$ Note that $\mathcal U \cap \mathcal I$ is connected because it is a linear subspaces of the affine spaces $\mathcal U$ (cf. Lemma \ref{eqn for I}). Thus $(\bar g \cdot \mathcal U) \cap \mathcal I = \bar g (\mathcal U \cap \mathcal I) $ is also connected. Since $\mathcal U \cap \mathcal I$ and $(\bar g \cdot \mathcal U) \cap \mathcal I$ share a common $k$-point, namely $x_0$, we see that that $f(\mathcal U\cap \mathcal I)$ and $f((\bar g \cdot \mathcal U)\cap \mathcal I )$ are in one connected component of $\OGr(d-1, d)$. We have $y \in \mathcal U \cap \mathcal I$ and $\bar g y \in (\bar g \cdot \mathcal U)\cap \mathcal I $. In particular $f(y)$ and $ f(\bar g y) $ are $R$-points of the aforementioned connected component of $\OGr(d-1, d)$. Recall from \cite[\S 3.2]{HPGU22} that $\OGr(d-1, d)$ has two connected components, and each is isomorphic to $\OGr(d-1)$ via the projection to the first factor. Our assumptions imply that $f(y), f(\bar g y)$ have the same image in $\OGr(d-1)$. It follows that $f(y) = f(\bar g y)$. But by definition $f$ is injective on $R$-points, so $y = \bar g y$.

\end{proof}
\begin{prop}\label{structure of R}
	Assume $x_0 \in S_{\Lambda}^{\bar g}(k)$. Then the local ring $\mathcal R = \oo_{\mathcal I^{\bar g}, x_0}$ of $\mathcal I ^{\bar g}$ at $x_0$, is isomorphic to the local ring at the origin of the  subscheme of $\adele^{d-1}_k$ defined by the equations (\ref{key equation}), where $\adele_k^{d-1}$ has coordinates $r_1,\cdots , r_{d-1}$. Moreover, explicitly we have 
	$$ \mathcal R \cong k[X]/ X^c.$$
\end{prop}
\begin{proof}
	The first claim follows from Lemma \ref{eqn for I} and Proposition \ref{prop: fixed pts in Y}. To compute $\mathcal R$ explicitly, we may and shall assume that the bases chosen in \ref{defn of U and I} are such that the matrix $H_1$ is already in its (upper-triangular) Jordan normal form. Recall from Definition \ref{defn of lambda c} that all the Jordan blocks have distinct eigenvalues. Let $J_{d_1} (\lambda_1), \cdots, J_{d_{s-1}} (\lambda_{s-1})$ be the Jordan blocks that have eigenvalues different from $\lambda$. Let $\lambda_s = \lambda$ and let $J_{d_s} (\lambda_s)$ be the Jordan block of eigenvalue $\lambda_s$ that appears in $H_1$, where we allow $d_s =0$. Then $d_s =c-1$. Moreover, we assume that $J_{d_1} (\lambda_1),\cdots, J_{d_s} (\lambda_s)$ appear in the indicated order. Note that $H_4 =\lambda$. 
	Write $H_1 = (h_{ij})_{1\leq i,j \leq d-1}$. The equations (\ref{key equation}) become 
	
\begin{align}\label{first eqn}
\begin{cases}
 r_{i-1} h_{i-1, i}  + ( h_{i,i}- \lambda +\vec r H_2) r_i = 0 ,~ 2\leq i \leq d-1 \\
( h_{1,1}- \lambda +\vec r H_2) r_1 = 0  
\end{cases}
\end{align}
		Note that when $h_{i,i}$ is not in the Jordan block $J_{d_s} (\lambda_s)$, we have $h_{i,i} -\lambda \in k^{\times}$, so the element $h_{i,i} - \lambda + \vec r H_2$ is a unit in the local ring $\oo_{\adele^{d-1}, 0}$. Hence for $i \leq d_1 + d_2 + \cdots + d_{s-1} = d-c$, each $r_i$ is solved to be a multiple of $r_{i-1}$ and this multiple eventually becomes zero when this procedure is iterated. In other words, the ideal in $\oo_{\adele^{d-1}, 0}$ defining $\mathcal R$ is generated by $$ r_1, r_2,\cdots, r_{d-c}, \quad (\vec r H_2)r_{d-c+1}, \quad   (\vec r H_2) r_i + r_{i-1}\ (d-c+1 < i \leq d-1).   $$
		When $c=1$, we have $\mathcal R \cong k$ as expected. Assume now $c \geq 2$. Let $h_1,\cdots, h_{c-1}$ be the last $c-1$ entries of the $(d-1)\times 1$-matrix $H_2$. Make the change of variables 
		$$\begin{cases}
X_i = r_{d-c +i}, ~ 1\leq i \leq c-1 , \\ 
A  = \vec r H_2.
		\end{cases}$$
		Then we have  
		$$\mathcal R \cong \left(\frac{k [X_1,\cdots, X_{c-1}, A] } { (  A-\sum_{ i =1}^{c-1} h_i X_i, ~ AX_1, X_1 + AX_2, ~X_2 +AX_3,\cdots, X_{c-2} +AX_{c-1})}\right)_{(X_1,\cdots, X_{c-1})}$$
		By eliminating the variables $X_1,\ldots,X_{c-2}$, we obtain that $$\mathcal{R} \cong \left(\frac{k[X_{c-1}, A]} { (X_{c-1} A^{c-1},~ A - X_{c-1} \sum_{i=0}^{c-2} h_{c-1-i} (-A)^i )}\right) _{(X_{c-1}, A)}.$$  
		Note that if $h_{c-1} = 0$, then the last two rows of the matrix $$\lambda I_d - \begin{pmatrix}
		H_1 & H_2 \\ 0 & H_4
		\end{pmatrix}$$ are both zero. This contradicts with the fact that the matrix $ \begin{pmatrix}
		H_1 & H_2 \\ 0 & H_4
		\end{pmatrix}$, which represents $\bar g$ on $\Phi \LL_d$, has in its Jordan normal form a unique Jordan block of eigenvalue $\lambda$ (cf. \S \ref{take x_0}).
		Hence $h_{c-1} \neq  0 $, and $\sum_{i=0}^{c-2} h_{c-1-i} (-A)^i$ is a unit in $k[X_{c-1}, A] _{(X_{c-1}, A)}$. It follows that
		$$\mathcal R\cong  \left(\frac{k[X]} { (X^c  )}\right) _{(X)} = k[X]/X^c,$$ as desired.
\end{proof}

\subsection{The intersection length formula}

We are now ready to determine the structure of the complete local ring of $S_{\Lambda}^{\bar g}$ at a $k$-point of it, when $p$ is large enough. It is a consequence of Lemma \ref{loc mod}, Proposition \ref{structure of R}, and some commutative algebra. 
\begin{thm}\label{main thm for multiplicity}
	Let $x_0 \in S_{\Lambda} ^{\bar g}(k)$. Let $\lambda $ and $c$ be as in Definition \ref{defn of lambda c}. Assume $p>c$. Then the complete local ring of $S_{\Lambda}^{\bar g}$ at $x_0$ is isomorphic to $k[X]/ X^c$. 
\end{thm}
\begin{proof}
Since $S_{\Lambda}$ is smooth of dimension $d-1$ (cf. \S \ref{sec:variety-s_lambda}), the complete local ring of $S_{\Lambda}^{\bar g}$ at $x_0$ is of the form $$ \hat {\mathcal S}= k[[X_1,\cdots, X_{d-1}]]/I $$ for a proper ideal $I$ of $k[[X_1,\cdots, X_{d-1}]]$.\footnote{We use this notation because previously we used the notation $\mathcal S$ to denote the local ring of $S_{\Lambda}^{\bar g} $ at $x_0$.} Let $\mathfrak m$ be the maximal ideal of $k [[X_1,\cdots, X_{d-1}]]$ and let $\bar {\mathfrak  m }$ be the maximal ideal of $\hat {\mathcal S}$. By Lemma \ref{loc mod} and Proposition \ref{structure of R}, there is an isomorphism
$$\beta: \hat {\mathcal S}/\bar {\mathfrak m}^p \isom k[X]/ X^c. $$

We first notice that if $R_1$ is any quotient ring of $k[[X_1,\cdots, X_{d-1}]]$ with its maximal ideal $\mathfrak m_1$ satisfying $\mathfrak m_1 = \mathfrak m_1^2$ (i.e. $R_1 $ has zero cotangent space), then $R_1 =k$. In fact, $R_1$ is noetherian and we have $\mathfrak m_1^l = \mathfrak m_1 $ for all $l\in \ZZ_{\geq 1}$, so by Krull's intersection theorem we conclude that $\mathfrak m_1 = 0$ and $R_1 =k$. 

Assume $c=1$. Then $\hat {\mathcal S}/ \bar{\mathfrak m} ^p \cong k$, so $\hat {\mathcal{S}}$ has zero cotangent space and thus $\hat {\mathcal S }= k$. Next we treat the case $c\geq 2$. Let $\alpha$ be the composite 
 $$\alpha: k [[ X_1,\cdots, X_{d-1}]] \to\hat { \mathcal S}/\bar {\mathfrak m}^p  \xrightarrow{\beta}  k[X]/X^c. $$ Let $J = \ker \alpha$. It suffices to prove that $I= J$. Note that because $\beta$ is an isomorphism we have \begin{align}\label{first reln between I and J}
 I +  {\mathfrak m }^p = J.
 \end{align}
 In the following we prove $\mathfrak m^p \subset I$, which will imply $I= J$ and hence the theorem. The argument is a variant of \cite[Lemma 11.1]{RTZ}. 
 
 Let $Y \in k [[ X_1,\cdots, X_{d-1}]]$ be such that $\alpha(Y) =X$. Since $X$ generates the maximal ideal in $k[X]/X^c$, we have 
 \begin{align}\label{m=J+Y} \mathfrak m =  J +(Y).
 \end{align} Then by (\ref{first reln between I and J}) and (\ref{m=J+Y}) we have $\mathfrak m = I + (Y) + \mathfrak m ^p$,
 and so the local ring $k[[X_1,\cdots, X_{d-1}]]/ (I+(Y))$ has zero cotangent space. We have observed that the cotangent space being zero implies that the ring has to be $k$, or equivalently
  \begin{align}\label{m=I+Y}
 \mathfrak m = I + (Y)  
 \end{align}
 
 Now we start to show $\mathfrak m^p \subset I$. By (\ref{m=I+Y}) we have $\mathfrak m^p \subset I+ (Y^p)$, so we only need to prove $Y^p\in  I$. We will show the stronger statement that $Y^c \in I$. By Krull's intersection theorem, it suffices to show that $Y^c \in I + \mathfrak m ^{pl}$ for all $l\geq 1$. In the rest we show this by induction on $l$.
 
  Assume $l=1$. Note that $\alpha (Y^c) =0$, so by (\ref{first reln between I and J} ) we have $$ Y^c \in J=  I + \mathfrak m^p.$$ Suppose $Y^ c \in  I + \mathfrak m^{pl}$ for an integer $ l \geq 1$. Write
   \begin{align}\label{decomp Y^c}
   Y^c = i +  m, ~i \in I,~ m \in \mathfrak m^{pl}.
   \end{align}   By (\ref{m=J+Y}) we know $$\mathfrak m^{pl} \subset ( J + (Y)) ^{pl} \subset \sum _{s=0}^{pl} J^s (Y) ^{pl-s}. $$ Thus we can decompose $m\in \mathfrak m ^{pl}$ into a sum 
   \begin{align}\label{decomp m}
    m = \sum_{s=0} ^{pl} j_s Y^{pl-s},~ j_s \in J^s.
   \end{align}
By (\ref{decomp Y^c}) and (\ref{decomp m}), we have 
\begin{align*}
Y^c = i +  \sum _{s=0}^{pl} j_s Y ^{pl-s} .
\end{align*}
 Splitting the summation $\sum_{ s = 0} ^{pl} $ into two sums $\sum_{s= 0} ^{pl-c}$ and $\sum_{ s = pl-c +1} ^{pl}$ and moving the sum $\sum_{s= 0} ^{pl-c}$ to the left hand side, we obtain
 \begin{align}\label{to extract A}
 Y^c - \sum _{s=0}^{pl-c} j_s Y ^{pl-s} =  i + \sum _{s=pl-c+1}^{pl} j_s Y ^{pl-s}.
 \end{align}
 
   Denote $$A : = \sum_{s =0 } ^{pl-c} j_s Y^{pl-s-c} .$$ Then the left hand side of (\ref{to extract A}) is equal to $(1-A) Y^c$. Hence we have $$(1-A)  Y^c = i + \sum_{ s= pl-c+1} ^{pl} j_s Y^{pl -s}  \subset I + J^{pl-c+1} \xlongequal{(\ref{first reln between I and J})} I+ (I+\mathfrak m^p) ^{pl-c+1} = I + \mathfrak m ^{p (pl-c+1)}  \subset I + \mathfrak m ^{p(l+1)}, $$ where for the last inclusion we have used $c< p$. Since $1-A$ is a unit in $k[[X_1,\cdots, X_{d-1}]] $ (because $c<p$), we have $Y^c \in I + \mathfrak m^{p(l+1)}$. By induction, $Y^c \in I + \mathfrak m^{pl}$ for all $l\in \ZZ_{\geq 1}$, as desired.
\end{proof}

\begin{cor}\label{cor:mainformula}Let $g \in J_b(\QQ_p)$ be regular semisimple and minuscule. Assume $\RZ^g \neq \varnothing$ and keep the notation of \ref{notation}. Let $x_0 \in (\delta(\RZ^\flat)\cap \RZ^g)(k)$. Let $(\LL_{d-1}, \LL_d) \in S_{\Lambda} (k)$ correspond to $x_0$ via Proposition \ref{prop:nonreduced} and define $\lambda, c$ as in Definition \ref{defn of lambda c}. Assume $p>c$. Then the complete local ring of $\delta(\RZ^\flat)\cap \RZ^g$ at $x_0$ is isomorphic to $k[X]/ X^c$. Moreover, we have $c=\frac{m(Q(T))+1}{2}$, where $Q(T)$ as in Theorem \ref{thm:pointscounting}. In particular, $1\leq c\leq n/2$. 
\end{cor}

\begin{proof}
The first part follows immediately from Proposition \ref{prop:nonreduced} and Theorem \ref{main thm for multiplicity}. It remains to show that $$c=\frac{m(Q(T))+1}{2}.$$ Suppose $x_0\in \BT_{\Lambda'}$ for some vertex lattice $\Lambda'$ (not necessarily equal to $\Lambda = L(g)^\vee$). Let $L$ be the associated special lattice. Then we have (\S \ref{sec:spec-latt-vert}) $$(\Lambda')_W^\vee\subseteq L\subseteq \Lambda'_W,\quad (\Lambda')_W^\vee\subseteq \Phi(L)\subseteq \Lambda'_W.$$ Hence the eigenvalue $\lambda$ of $\bar g$ on $\Phi(\mathcal{L}_d)/\mathcal{L}_{d-1}\cong (L+\Phi(L))/L$ appears among the eigenvalues of $\bar g$ on $\Lambda' /(\Lambda')^\vee$, and so the minimal polynomial of $\bar g$ on $\Lambda' /(\Lambda')^\vee$ in $\mathbb{F}_p[T]$ is equal to $Q(T)$ by the proof of Theorem \ref{thm:pointscounting}. Notice that the characteristic polynomial of $\bar g$ on $\Phi(\mathcal{L}_d)$ (in $k[T]$) divides  $R(T)Q(T)$ (the characteristic polynomial of $\bar g$ on $\Lambda'_W/L(g)$) and also is divided by $R(T)$ (the characteristic polynomial of $\bar g$ on $(\Lambda')_W^\vee/L(g)$). It follows that $c$, the multiplicity of $\lambda$ of  $\bar g$ on $\Phi(\mathcal{L}_d)$, is equal to the multiplicity of $\lambda$ in $R(T)Q(T)$. The desired formula for $c$ then follows since $$m(Q(T))+1=2\cdot\text{ the multiplicity of } Q(T)\text{ in }R(T)Q(T).$$ Finally, we note that $m(Q(T))$ is a positive odd integer not greater than the degree of $P(T)$, and the latter, being the type of the vertex lattice $\Lambda = L(g)^{\vee}$, is an even integer $\leq t_{\max}$ (cf. \S \ref{sec:vertex-lattices}). The bound for $c$ follows from the value of $ t_{\max}$ given in \S \ref{sec:vertex-lattices}. 
\end{proof}

\bibliographystyle{hep}
\bibliography{myref}

\begin{thebibliography}{{Kim}13}

\bibitem[BP17]{Bueltel2017}
O.~{Bueltel} and G.~{Pappas}, \textsl{ {(G, $\mu$)-displays and Rapoport-Zink
  spaces}},
\newblock ArXiv e-prints  (February 2017), {1702.00291}.

\bibitem[DL76]{Deligne1976}
P.~Deligne and G.~Lusztig, \textsl{ Representations of reductive groups over
  finite fields},
\newblock Ann. of Math. (2) \textbf{ 103}(1), 103--161 (1976).

\bibitem[GGP12]{Gan2012}
W.~T. Gan, B.~H. Gross and D.~Prasad, \textsl{ Symplectic local root numbers,
  central critical {$L$} values, and restriction problems in the representation
  theory of classical groups},
\newblock Ast\'erisque (346), 1--109 (2012),
\newblock Sur les conjectures de Gross et Prasad. I.

\bibitem[GK92]{Gross1992}
B.~H. Gross and S.~S. Kudla, \textsl{ Heights and the central critical values
  of triple product {$L$}-functions},
\newblock Compositio Math. \textbf{ 81}(2), 143--209 (1992).

\bibitem[GS95]{Gross1995}
B.~H. Gross and C.~Schoen, \textsl{ The modified diagonal cycle on the triple
  product of a pointed curve},
\newblock Ann. Inst. Fourier (Grenoble) \textbf{ 45}(3), 649--679 (1995).

\bibitem[GZ86]{Gross1986}
B.~H. Gross and D.~B. Zagier, \textsl{ Heegner points and derivatives of
  {$L$}-series},
\newblock Invent. Math. \textbf{ 84}(2), 225--320 (1986).

\bibitem[Har95]{Joe}
J.~Harris,
\newblock \textsl{ Algebraic geometry}, volume 133 of \textsl{ Graduate Texts
  in Mathematics},
\newblock Springer-Verlag, New York, 1995,
\newblock A first course, Corrected reprint of the 1992 original.

\bibitem[HP14]{HPGU22}
B.~Howard and G.~Pappas, \textsl{ On the supersingular locus of the {${\rm
  GU}(2,2)$} {S}himura variety},
\newblock Algebra Number Theory \textbf{ 8}(7), 1659--1699 (2014).

\bibitem[HP15]{Howard2015}
B.~{Howard} and G.~{Pappas}, \textsl{ {Rapoport-Zink spaces for spinor
  groups}},
\newblock ArXiv e-prints  (September 2015), {1509.03914}.

\bibitem[Ive72]{Iversen1972}
B.~Iversen, \textsl{ A fixed point formula for action of tori on algebraic
  varieties},
\newblock Invent. Math. \textbf{ 16}, 229--236 (1972).

\bibitem[{Kim}13]{Kim2013}
W.~{Kim}, \textsl{ {Rapoport-Zink spaces of Hodge type}},
\newblock ArXiv e-prints  (August 2013), {1308.5537}.

\bibitem[Kis10]{kisin2010integral}
M.~Kisin, \textsl{ Integral models for {S}himura varieties of abelian type},
\newblock J. Amer. Math. Soc. \textbf{ 23}(4), 967--1012 (2010).

\bibitem[Kri16]{Krishna2016}
R.~M. Krishna,
\newblock \textsl{ Relative {T}race {F}ormula for {SO}2 x {SO}3 and the
  {W}aldspurger {F}ormula},
\newblock ProQuest LLC, Ann Arbor, MI, 2016,
\newblock Thesis (Ph.D.)--Columbia University.

\bibitem[Lus11]{Lusztig2011}
G.~Lusztig, \textsl{ From conjugacy classes in the {W}eyl group to unipotent
  classes},
\newblock Represent. Theory \textbf{ 15}, 494--530 (2011).

\bibitem[Lus77]{Lusztig1976/77}
G.~Lusztig, \textsl{ Coxeter orbits and eigenspaces of {F}robenius},
\newblock Invent. Math. \textbf{ 38}(2), 101--159 (1976/77).

\bibitem[MP16]{MPspin}
K.~Madapusi~Pera, \textsl{ Integral canonical models for spin {S}himura
  varieties},
\newblock Compos. Math. \textbf{ 152}(4), 769--824 (2016).

\bibitem[RTZ13]{RTZ}
M.~Rapoport, U.~Terstiege and W.~Zhang, \textsl{ On the arithmetic fundamental
  lemma in the minuscule case},
\newblock Compos. Math. \textbf{ 149}(10), 1631--1666 (2013).

\bibitem[RZ96]{RZ96}
M.~Rapoport and T.~Zink,
\newblock \textsl{ Period spaces for {$p$}-divisible groups}, volume 141 of
  \textsl{ Annals of Mathematics Studies},
\newblock Princeton University Press, Princeton, NJ, 1996.

\bibitem[YZZ]{Yuan}
X.~Yuan, S.-W. Zhang and W.~Zhang, \textsl{ {Triple product L-series and
  Gross-Kudla-Schoen cycles}},
\newblock preprint .

\bibitem[YZZ13]{Yuan2013}
X.~Yuan, S.-W. Zhang and W.~Zhang,
\newblock \textsl{ The {G}ross-{Z}agier formula on {S}himura curves}, volume
  184 of \textsl{ Annals of Mathematics Studies},
\newblock Princeton University Press, Princeton, NJ, 2013.

\bibitem[Zha12]{Zhang2012}
W.~Zhang, \textsl{ On arithmetic fundamental lemmas},
\newblock Invent. Math. \textbf{ 188}(1), 197--252 (2012).

\end{thebibliography}
	
\end{document}